\documentclass[12pt]{amsart}
\usepackage[latin1]{inputenc}
\usepackage{amssymb, amsmath, amscd, mathrsfs,marvosym, psfrag,color,a4} %, txfonts} 

\input xy
\xyoption{all}
\DeclareMathAlphabet{\mathpzc}{OT1}{pzc}{m}{it}

\newtheorem{theorem}{Theorem}[section]

\newtheorem{proposition}[theorem]{Proposition}
\newtheorem{corollary}[theorem]{Corollary}

\newtheorem{lemma}[theorem]{Lemma}

 \newtheorem*{GKM}{The GKM condition}

\theoremstyle{definition}
\newtheorem{definition}[theorem]{Definition}

\theoremstyle{remark}
\newtheorem{remark}[theorem]{Remark}

\newcommand{\CA}{{\mathcal A}}

\newcommand{\CC}{{\mathcal C}}

\newcommand{\CG}{{\mathcal G}}

\newcommand{\CI}{{\mathcal I}}
\newcommand{\CJ}{{\mathcal J}}
\newcommand{\CK}{{\mathcal K}}
\newcommand{\CL}{{\mathcal L}}

\newcommand{\CO}{{\mathcal O}}

\newcommand{\CS}{{\mathcal S}}
\newcommand{\CT}{{\mathcal T}}

\newcommand{\CV}{{\mathcal V}}
\newcommand{\CW}{{\mathcal W}}

\newcommand{\CX}{{\mathcal X}}

\newcommand{\CY}{{\mathcal Y}}
\newcommand{\CZ}{{\mathcal Z}}

\newcommand{\SA}{{\mathscr A}}
\newcommand{\SB}{{\mathscr B}}
\newcommand{\SC}{{\mathscr C}}
\newcommand{\SD}{{\mathscr D}}

\newcommand{\SM}{{\mathscr M}}
\newcommand{\SN}{{\mathscr N}}
\newcommand{\SO}{{\mathscr O}}
\newcommand{\SP}{{\mathscr P}}
\newcommand{\SQ}{{\mathscr Q}}

\newcommand{\SV}{{\mathscr V}}

\newcommand{\SX}{{\mathscr X}}
\newcommand{\SY}{{\mathscr Y}}

\newcommand{\fg}{{{\mathfrak g}}}

\newcommand{\hCW}{{\widehat\CW}}

\newcommand{\hCS}{{\widehat\CS}}

\newcommand{\tCJ}{{\widetilde{\CJ}}}

\newcommand{\DR}{{\mathbb R}}
\newcommand{\DZ}{{\mathbb Z}}

\newcommand{\DV}{{\mathbb V}}

\newcommand{\DT}{{\mathbb T}}

\newcommand{\bB}{{\mathbf B}}
\newcommand{\bC}{{\mathbf C}}

\newcommand{\bK}{{\mathbf K}}

\newcommand{\bM}{{\mathbf M}}

\newcommand{\bP}{{\mathbf P}}

\newcommand{\bS}{{\mathbf S}}

\newcommand{\bZ}{{\mathbf Z}}

\newcommand{\Hom}{{\operatorname{Hom}}}

\newcommand{\supp}{{\operatorname{supp}}}

\newcommand{\rk}{{{\operatorname{rk}}}}

\newcommand{\ol}{\overline}

\newcommand{\id}{{\operatorname{id}}}

\newcommand{\lgl}{\langle}
\newcommand{\rgl}{\rangle}
\newcommand{\comment}[1]{}

\begin{document}

\pagenumbering{arabic}
\title[]{ Sheaves on the  alcoves  and modular representations II} 

\author[]{Peter Fiebig, Martina Lanini}
\begin{abstract} We relate the category of sheaves on alcoves that was constructed in \cite{FieLanWallCross} to the representation theory of reductive algebraic groups. In particular, we show that its indecomposable projective objects encode the simple rational characters of a reductive algebraic group for characteristics above the Coxeter number. 
\end{abstract}

\address{Department Mathematik, FAU Erlangen--N\"urnberg, Cauerstra\ss e 11, 91058 Erlangen}
\email{fiebig@math.fau.de}
\address{Universit\`a degli Studi di Roma ``Tor Vergata", Dipartimento di Matematica, Via della Ricerca Scientifica 1, I-00133 Rome, Italy }
\email{lanini@mat.uniroma2.it}
\maketitle
%\tableofcontents

\section{Introduction}
 
Let $R$ be a finite irreducible root system and denote by $\CA$ the associated set of (affine) alcoves. In the article \cite{FieLanWallCross} we endowed $\CA$ with a topology and defined a category $\bS$ of sheaves of $\CZ$-modules on $\CA$, where $\CZ$ is the structure algebra of $R$ over a field $k$.  In this article we relate $\bS$  to the representation theory of algebraic groups, i.e. we show that if $k$ is algebraically closed and if its characteristic is bigger than the Coxeter number of $R$, then one can calculate the irreducible characters of any reductive algebraic group $G$ over $k$ with root system $R$ from data in $\bS$. 

Here is what we show in this article in more precise terms. For any alcove $A\in\CA$ we construct a ``standard object'' $\SV(A)$ in $\bS$. As $\bS$ is a full  subcategory of the abelian category of sheaves of $\CZ$-modules on $\CA$, it inherits an exact structure, i.e. a notion of short exact sequences.  We consider the full subcategory $\bB$ of $\bS$ that contains all objects that admit a finite filtration with subquotients isomorphic to various $\SV(A)$, possibly shifted in degree. We call $\bB$ the category of {\em objects admitting a Verma flag}. We then show that for each alcove $A$ there exists a unique indecomposable projective object $\SB(A)$ in $\bS$ that admits an epimorphism onto $\SV(A)$. Moreover, we prove that each $\SB(A)$ admits a Verma flag. The occurence of $\SV(A)$ as a subquotient in a filtration of $\SB(B)$ is independent of the choice of the filtration and is denoted by $(\SB(B):\SV(A))$. As the category $\bS$ admits a grading, this is actually a polynomial in $\DZ[v,v^{-1}]$.

Now suppose that $k$ is algebraically closed with characteristic $p$ larger than the Coxeter number of $R$. Let $G$ be the connected,  almost simple, simply connected algebraic group associated with $R$ over  $k$. Fix a maximal torus $T$ in $G$ and denote by  $X$ its character lattice. Denote by  $\fg$  the Lie algebra of $G$. Consider the category  $\bC$ of restricted, $X$-graded representations of $\fg$ as it is defined in \cite{Soe95}.  In this category one has standard objects $Z(\lambda)$ (the baby Verma modules) and projective objects $Q(\mu)$, and both sets are parametrized by  elements in $X$. Each projective object admits a finite fitration with subquotients being isomorphic to baby Verma modules, and again the multiplicity $(Q(\mu):Z(\lambda))$ is independent of choices. One can obtain the irreducible rational characters of $G$ via known results from these multiplicities (cf. \cite{FieBull} for an overview). From the (rational) characters of $G$ one can deduce the characters of any reductive group with root system $R$ via classical results. %From these one obtains the simple rational characters for any reductive algebraic group over $k$ with root system $R$. 

It is even sufficient to know the numbers $(Q(\mu):Z(\lambda))$  in the case that $\mu+\rho$ and $\lambda+\rho$ are contained in a regular orbit of the affine Weyl group $\hCW_p:=\CW\ltimes p\DZ R$ acting  on $X$, where $\rho\in X$ is the half-sum of positive roots and $\CW$ is the finite Weyl group. Let us fix such a regular orbit $\SO$ in $X$. Then there is a natural bijection $\sigma\colon\SO\xrightarrow{\sim}\CA$ that associates to $\gamma\in\SO$ the unique alcove that contains $\gamma/p$.  The main result in this article is that for $\lambda,\mu\in\SO$ one  has 
$$
(Q(\mu-\rho):Z(\lambda-\rho))=(\SB(\sigma(\mu)):\SV(\sigma(\lambda)))(1).
$$ Note that on the right hand side $(\cdot)(1)$ denotes the evaluation of a Laurent polynomial at $1\in\DZ$.  

The above result is obtained via the categorical equivalence established by Andersen, Jantzen and Soergel in their seminal work on Lusztig's conjecture \cite{AJS}.  However, we hope that the new approach developed in this article and its companion article \cite{FieLanWallCross} will ultimately result in a simplified proof of the categorical equivalence of Andersen, Jantzen and Soergel. 

\subsection{Acknowledgements}
Both authors were partially supported by the DFG grant SP1388. The second author acknowledges the MIUR Excellence Department Project awarded to the Department of Mathematics, University of Rome Tor Vergata, CUP E83C18000100006. We would like to thank an anonymous referee for pointing out gaps in an earlier version of this manuscript.

\section{Topologies on the space of alcoves and the structure algebra}
In this section we quickly review the topology on the set of alcoves that we considered in \cite{FieLanWallCross}.  We add some new results that were not needed in that article. In particular, we study sections of the action of the root lattice on the set of alcoves. We also recall the definition of the structure algebra of our datum.

\subsection{Alcoves}\label{subsec-alcoves}

Let $V$ be a real vector space and $R\subset V$  a finite irreducible root system with a system of positive roots $R^+\subset R$. Denote by $\alpha^\vee\in V^{\ast}=\Hom_\DR(V,\DR)$ the coroot associated with $\alpha\in R$. Let
% and denote by $\Delta\subset R^+$ the set of simple roots.  The positive coroots are then given by  $R^{\vee,+}=\{\alpha^\vee\in R^\vee\mid \alpha\in R^+\}$, and the simple coroots by $\Delta^\vee=\{\alpha^\vee\in R^{+,\vee}\mid \alpha\in\Delta\}$. We define 
$X:=\{\lambda\in V\mid \langle \lambda,\alpha^\vee\rangle\in\DZ\text{ for all $\alpha\in R$}\}$ and $X^\vee:=\{v\in V^\ast\mid \langle\alpha,v\rangle\in\DZ\text{ for all $\alpha\in R$}\}$ be the weight and the coweight lattice, resp. 
For $\alpha\in R^+$ and $n\in\DZ$ define 
$$
 H_{\alpha,n}:=\{\mu\in V\mid \langle \mu,\alpha^\vee\rangle = n\} \text{ and }
H_{\alpha,n}^\pm:=\{\mu\in V\mid \pm\langle \mu, \alpha^\vee\rangle>\pm n\},
$$
where $\langle\cdot,\cdot\rangle\colon V\times V^\ast\to \DR$ is the natural pairing. 
The connected components of $V\setminus\bigcup_{\alpha\in R^{+},n\in\DZ}H_{\alpha,n}$ are called {\em alcoves}. 
Denote by $\CA$ the set of alcoves.

  Denote by $s_{\alpha,n}\colon V\to V$, $\lambda\mapsto \lambda-(\langle \lambda,\alpha^\vee\rangle-n)\alpha$ the affine reflection with fixed point hyperplane $H_{\alpha,n}$. 
 The {\em affine Weyl group} is the group $\hCW$ of affine transformations on $V$ generated by the set $\{s_{\alpha,n}\}_{\alpha\in R^+,n\in\DZ}$.  
 It contains the translations $t_\gamma\colon V\to V$, $\mu\mapsto \mu+\gamma$ for $\gamma\in \DZ R$. 
The action of $\hCW$ on $V$ preserves the set of hyperplanes $H_{\alpha,n}$ and induces an action on the set $\CA$. For $\alpha\in R^+$ denote by $s_{\alpha}=s_{\alpha,0}$ the ($\DR$-linear) reflection at the hyperplane $H_{\alpha,0}$, and let $\CW\subset\hCW$ be the subgroup generated by the set $\{s_\alpha\}_{\alpha\in R^+}$. This subgroup is called the {\em finite Weyl group}, and $\hCW=\CW\ltimes\DZ R$.

\subsection{Base rings}

Let $k$ be a field and denote by $X^\vee_k=X^\vee\otimes_\DZ k$ the $k$-vector space associated with the lattice $X^\vee$. For  $v\in X^\vee$ we  denote by $v$ its canonical image  $v\otimes 1$ in $X^\vee_k$. We assume that $k$ satisfies the following condition. 

\begin{GKM} 
The characteristic of $k$ is not $2$ and   $\alpha^\vee\not\in k\beta^\vee$ in $X_k^\vee$ for all $\alpha,\beta\in R^+$, $\alpha\ne \beta$.
\end{GKM}

Let $S=S(X^\vee_k)$ be the symmetric algebra of the $k$-vector space $X^\vee_k$. Let $T$ be a unital, commutative and flat  $S$-algebra. Then the left multiplication with $\alpha^\vee$ on $T$ is non-zero for all $\alpha\in R^+$. 

\begin{definition}
\begin{enumerate}
\item  Let $I_T\subset R^+$ be the set of all $\alpha$ such that $\alpha^\vee$ is {\em not} invertible in $T$. 
\item Denote by $\hCW_T$  the subgroup of $\hCW$ generated by the sets $\{t_\gamma\}_{\gamma\in\DZ R}$ and $\{s_{\alpha,n}\}_{\alpha\in I_T, n\in\DZ}$. \item Let $R_T^+$ be the set of all $\alpha\in R^+$ such that $\hCW_T$ contains $s_{\alpha,n}$ for some (equivalently, all) $n\in\DZ$.\item  We call $T$ {\em saturated} if $I_T=R_T^+$.
\item  A {\em base ring} is a unital, commutative, flat, saturated $S$-algebra $T$ with the property that any projective $T$-module is free. We say that $T$ is {\em generic}, if $R_T^+=\emptyset$, and {\em subgeneric}, if $R_T^+=\{\alpha\}$ for some $\alpha\in R^+$. 
\end{enumerate}
\end{definition}
If $T$ is generic, then $\hCW_T=\DZ R$. If $T$ is subgeneric, then $\hCW_T=\{\id_V,s_{\alpha}\}\ltimes\DZ R$ for $\alpha\in R_T^+$. 
Consider $S$ as an evenly graded algebra with  $S_2=X_k^\vee$.
We will sometimes assume that a base ring  $T$ is, in addition,  a graded $S$-algebra (the most important case is $T=S$). In this case, all $T$-modules are assumed to be graded, and homomorphisms between $T$-modules are assumed to be graded of degree $0$. For a graded module $M=\bigoplus_{n\in\DZ} M_n$ and $l\in\DZ$ we define the shifted graded $T$-module $M[l]$ by $M[l]_n=M_{l+n}$.

\subsection{Topologies on $\CA$}
Fix  a base ring $T$.

%Let $T\subset T^\prime$ be an inclusion of base rings. 
%For a $T$-module $M$ we set $M^{T^\prime}:=M\otimes_TT^\prime$. In the case $T^\prime=S^\emptyset$ we write $M^\emptyset$ instead of $M^{S^\emptyset}$ and in the case $T^\prime=T^\alpha$ we write $M^{\alpha}$ instead of $M^{T^\alpha}$. The following associates a subset of positive roots to any localization.
%
 \begin{definition} \begin{enumerate}
 \item We denote by $\preceq_T$ the partial order on the set $\CA$ that is generated by $A\preceq_T B$, if either $B=t_\gamma(A)$ with $\gamma\in\DZ R$, $0\le\gamma$, or  if $B=s_{\alpha,m}(A)$, where $\alpha\in R_T^+$ and $m\in\DZ$ are such that $A\subset H_{\alpha,m}^-$.  
 \item A subset $\CJ$ of $\CA$ is called {\em $T$-open} if it is an order ideal with respect to $\preceq_T$, i.e. if $A\in\CJ$ and $B\preceq_T A$ imply $B\in\CJ$.
 \end{enumerate}
 \end{definition} 
 
 Here, $\le$ denotes the usual order on $\DZ R$, i.e. $\lambda\le\mu$ if and only if $\mu-\lambda$ is a sum of positive roots. 
 The collection of $T$-open subsets in $\CA$  clearly defines a topology, and we denote by $\CA_T$ the resulting topological space. Note that arbitrary intersections of open subsets in this topology are open again. We will often use the notation $\{\preceq_T A\}$ for the set $\{B\in\CA\mid B\preceq_T A\}$. The sets $\{\succeq_T A\}$, $\{\prec_T A\}$, etc. are defined similarly.

 For a subset $\CT$ of $\CA$ we set 
 $$
 \CT_{\preceq_T}:=\bigcup_{A\in\CT}\{\preceq_T A\}\text{ and } \CT_{\succeq_T}:=\bigcup_{A\in\CT}\{\succeq_T A\}.
 $$
 Then $\CT_{\preceq_T}$ is the smallest $T$-open subset of $\CA$ that contains $\CT$, and $\CT_{\succeq_T}$ is the smallest $T$-closed subset  of $\CA$ that contains $\CT$. 

 \begin{lemma}\label{lemma-top} \begin{enumerate}
 \item A subset $\CK$ of $\CA_T$ is locally closed if and only if  $A,C\in\CK$ and $A\preceq_T B\preceq_T C$ imply $B\in\CK$.
 \item If $\CK\subset\CA_T$ is locally closed, then $\CK_{{\preceq_T}}\setminus\CK$ is open.
\end{enumerate} \end{lemma}
\begin{proof} Part (1) is an easy consequence of the definition. We prove (2). 
Let $A$ be an element in $\CK_{{\preceq_T}}\setminus\CK$. Then there is some $C\in\CK$ with $A{\preceq_T} C$. Suppose $B{\preceq_T} A$. Then $B\in\CK_{{\preceq_T}}$. Now $B\in\CK$ would imply   $A\in\CK$ by (1), hence $B\in\CK_{{\preceq_T}}\setminus\CK$.  
\end{proof}

We  now collect two results that were proven in \cite{FieLanWallCross}.
\begin{lemma}{\cite[Lemma 3.11]{FieLanWallCross}}
The connected components of $\CA_T$ coincide with the $\hCW_T$-orbits.
\end{lemma}
In particular, if  $T$ is a generic base ring, then $\hCW_T=\DZ R$, so  the connected components of $\CA_T$ are just the $\DZ R$-orbits. If $T$ is a subgeneric base ring, then every connected component of $\CA_T$ splits into two $\DZ R$-orbits $x$ and $y$. If $\alpha\in R^+$ is such that $R_T^+=\{\alpha\}$, then $y=s_{\alpha,0}x$ (cf. Lemma 3.13 in \cite{FieLanWallCross}). We will use lower case Latin letters such as $x,y,w$ for $\DZ R$-orbits in $\CA$. For arbitrary connected components or unions of components we use upper case Greek letters such as $\Lambda$ or $\Omega$. We denote by $C(\CA_T)$ the set of connected components of $\CA_T$.

\begin{lemma}{\cite[Lemma 3.10]{FieLanWallCross}}\label{lemma-indtop}
The topology on a $\DZ R$-orbit $x\subset\CA_T$ induced via the inclusion is independent of $T$. 
\end{lemma} 

More specifically, a subset $\CJ$ of $x$ is open if and only if $A\in\CJ$ and $\gamma\in\DZ R$, $\gamma\ge 0$ imply $A-\gamma\in \CJ$. 
\subsection{The map $\alpha\uparrow\cdot$}
Let $\alpha\in R^+$, and let $A\in\CA$. 
%Denote by $\Lambda$ the $\DZ R$-orbit of $A$. By Lemma \ref{lemma-easypeasy}, the alcoves $s_{\alpha,n}B$ with $B\in\Lambda$ are contained in a single $\DZ R$-orbit $\Lambda^\prime$. 
Let $m\in\DZ$ be minimal such that $A\subset  H_{\alpha,m}^-$. Define 
$$
\alpha\uparrow A:=s_{\alpha,m} A.
$$
Then $\alpha\uparrow(\cdot)$ is a bijection on the set $\CA$. We denote by $\alpha\downarrow(\cdot)$ its inverse. 

Suppose that $T$ is subgeneric with  $R_T^+=\{\alpha\}$. Let $\Lambda\subset\CA_T$ be a connected component that splits into the $\DZ R$-orbits $x$ and $y$. For $A\in x$ we then have $\alpha\uparrow A\in y$. 

\begin{lemma} \label{lemma-upmin}  Suppose that $T$ is subgeneric and let $\Lambda=x\cup y$ be as above. Let $A\in x$. Then  $\alpha\uparrow A$ is the unique $\preceq_T$-minimal alcove in $\{\succeq_TA\}\cap y$.  In particular, the set $\{A,\alpha\uparrow A\}$ is locally closed in $\CA_T$.
\end{lemma}
Note that for general $T$ it is not always true that $\alpha\uparrow A$ is the minimal element in its $\DZ R$-orbit with  $\alpha\uparrow A \succeq_T A$. 

\begin{proof}   Let $B$ be a minimal element in $\{\succeq_TA\}\cap y$. Then there is a sequence $A=A_0\preceq_T A_1\preceq_T\cdots\preceq_T A_n=B$ such that $A_i=A_{i-1}+\mu$ for some $\mu>0$ or $A_i=s_{\alpha,l}A_{i-1}$ for some $l\in\DZ$. The fact that $s_{\alpha,m}(x)=y$ and the minimality of $B$ imply that  the second case  occurs if and only if  $i=n$. Hence $B=s_{\alpha,l}(A+\mu)$ for some $l\in\DZ$ and $\mu\ge 0$.  The minimality  implies that $s_{\alpha,l}(A+\mu)=\alpha\uparrow(A+\mu)=(\alpha\uparrow A)+\mu$. Applying the minimality for the third time shows $\mu=0$.
\end{proof}

\subsection{Induced topologies and sections}\label{subsec-toponV} 

Fix $\mu\in X$ and let $\CK_\mu\subset\CA$ be the set of all alcoves that contain $\mu$ in their closure. Endow $\CK_\mu$ with the topology induced via the inclusion into $\CA_T$. For a connected component $\Lambda$ of $\CA_T$ set $\Lambda_\mu:=\Lambda\cap\CK_\mu$. Then $\CK_\mu=\bigcup_{\Lambda\in C(\CA_T)}\Lambda_\mu$ is the decomposition into connected components. 
Denote by $\CW_\mu\subset\hCW$ the stabilizer of $\mu$.  It is generated by the reflections $s_{\alpha,\lgl\mu,\alpha^\vee\rgl}$ with $\alpha\in R^+$, and  $\CK_\mu$ is a $\CW_\mu$-orbit. Set $\CW_{T,\mu}=\CW_\mu\cap\hCW_T$. As each connected component $\Lambda$ of $\CA_T$ is a $\hCW_T$-orbit, each connected component $\Lambda_\mu$ is a $\CW_{T,\mu}$-orbit. 

Now denote by $\CV$ the set of $\DZ R$-orbits in $\CA$, and let $\pi\colon\CA\to\CV$ be the orbit map. We often denote by $\ol A$ or $\ol\Lambda$ the image of $A\in\CA$ or $\Lambda\subset\CA$ under $\pi$.  Standard arguments show that $\hCW=\CW_\mu\ltimes\DZ R$ and $\hCW_T=\CW_{T,\mu}\ltimes\DZ R$. As $\CA$ is a principal homogeneous $\hCW$-set, it follows that $\pi|_{\CK_\mu}\colon\CK_\mu\to\CV$ is a bijection. We call the topology on $\CV$ obtained from the topology on $\CK_\mu$ considered above via this bijection the {\em $(T,\mu)$-topology}, and we denote the resulting topological space by $\CV_{T,\mu}$.  So a subset $\CL$ of $\CV$ is {\em $(T,\mu)$-open} if and only if there exists a $T$-open subset $\CJ$ of $\CA$ such that $\CL=\pi(\CK_\mu\cap\CJ)$. Note that this topology only depends on the $\DZ R$-orbit of $\mu$. From the above we obtain that  $\CV_{T,\mu}=\bigcup_{\Lambda\in C(\CA_T)}\pi(\Lambda_\mu)$ is the decomposition of $\CV_{T,\mu}$ into connected components. 

Now fix $\Lambda\in C(\CA_T)$, let $A\in\Lambda_\mu$ be a $\preceq_T$-minimal element, and denote by  $\CS_{T,\mu}\subset\CW_{T,\mu}$  the set of reflections at hyperplanes having a codimension $1$ intersection with the closure of $A$. Again, standard arguments show that $(\CW_{T,\mu},\CS_{T,\mu})$ is a Coxeter system (it is the Coxeter system associated with the root system $R_T^+$). We denote by $\le_{T,\mu}$ the corresponding Bruhat order on $\CW_{T,\mu}$.

\begin{lemma} \label{lemma-tauiso} Let $A\in\Lambda_\mu$ be as above. Then the map $\tau\colon \CW_{T,\mu}\to \Lambda_\mu$, $w\mapsto w(A)$ is an isomorphism of partially ordered sets. In particular, the element $A$ is the unique minimal element in $\Lambda_\mu$. 
\end{lemma}
\begin{proof} We have already argued that $\tau$ is a bijection.  The Bruhat order on $\CW_{T,\mu}$ is generated by the relations $w\le_{T,\mu} ws_{\alpha,\lgl\mu,\alpha^\vee\rgl}$ for $\alpha\in R_T^+$ and $w\in\CW_{T,\mu}$ such that  $w(\alpha)\in R_T^+$. Equivalently, it is generated by the relations $w\le_{T,\mu} s_{\alpha,\lgl\mu,\alpha^\vee\rgl} w$ for $\alpha\in R_T^+$ and $w\in\CW_{T,\mu}$ such that  $w^{-1}(\alpha)\in R_T^+$  Denote by $\lambda$ the barycenter of $A$. Then the minimality of $A$ implies $A\subset H_{\beta,\lgl\mu,\beta^\vee\rgl}^-$, i.e. $\lgl\lambda,\beta^\vee\rgl<\lgl\mu,\beta^\vee\rgl$ for any $\beta\in R_T^+$. Now recall that $w(\mu)=\mu$. Hence, as $\lgl w(\nu),\alpha^\vee\rgl=\lgl \nu, w^{-1}(\alpha)^\vee\rgl$ for all $\nu\in X$, we deduce $\lgl w(\lambda),\alpha^\vee\rgl < \lgl\mu,\alpha^\vee\rgl$, i.e. $w(A)\subset H_{\alpha,\lgl\mu,\alpha^\vee\rgl}^-$, if and only if $w^{-1}(\alpha)\in R_T^+$. Now $\lgl w(\lambda),\alpha^\vee\rgl < \lgl\mu,\alpha^\vee\rgl$  if and only if $w(A)\preceq_Ts_{\alpha,\lgl\mu,\alpha^\vee\rgl}w(A)$. As the restriction of $\preceq_T$ to $\Lambda_\mu$ is generated by these relations, the claim follows. 
\end{proof}

Let  $\CK$ be a subset of $\CA_T$. 
\begin{definition} We say that $\CK$ is a {\em section} if the following holds:
\begin{enumerate}
\item $\pi|_\CK\colon\CK\to\CV$ is injective.  
\item $\CK_{\preceq_T}=\bigcup_{\gamma\ge 0}\CK-\gamma$ and $\CK_{\succeq_T}=\bigcup_{\gamma\ge 0}\CK+\gamma$.
\end{enumerate}
\end{definition} 
If $\CK$ is a section, then properties (1) and (2) imply  $\CK=\CK_{\preceq_T}\cap\CK_{\succeq_T}$, so $\CK$ is locally closed.

%Now let $\mu\in X$ and $\Lambda$ a connected component of $\CA_T$. 
\begin{lemma}\label{lemma-specset1} The set $\Lambda_\mu$ is a section. 
\end{lemma}
\begin{proof} It is clear that $\pi|_{\Lambda_\mu}$ is injective. We need to prove the defining property (2) of a section. As $\Lambda$ is a connected component, the sets $(\Lambda_\mu)_{\preceq_T}$ and $(\Lambda_\mu)_{\succeq_T}$ are contained in $\Lambda$. Let $A\in\Lambda_\mu$ and suppose that $B\preceq_T A$. For a point $v_A\in A$ denote by $v_B$ the unique element in $B\cap \hCW_T(v_A)$. Then $v_A-v_B\in\DR_{\ge0}R^+$ by the arguments that we used in the proof of Lemma 3.10 in \cite{FieLanWallCross}. By letting $v_A$ move towards $\mu$ we obtain an element $\mu^\prime\in X$ in the closure of $B$ with $\mu-\mu^\prime\ge 0$. Hence $B\in\Lambda_{\mu^\prime}=\Lambda_\mu-(\mu-\mu^\prime)\subset \bigcup_{\gamma\ge 0}\Lambda_\mu-\gamma$. Similarly one shows $(\Lambda_\mu)_{\succeq_T}=\bigcup_{\gamma\ge0}\Lambda_\mu+\gamma$.
\end{proof}
%We call the sets of the form $\Lambda_\mu$ with $\lambda\in X$ the {\em special sections}. 
%The map $\hCW\to\CA$, $w\mapsto A_w:=w(A_0)$ is a bijection. In particular, we obtain an action of $\hCW$ on the right of $\CA$ by transport of structure. Explicitely, $A.w:=A_{xw}$ for $x,w\in\hCW$ if $A=A_x$. 
%Denote by $\hCS\subset\hCW$ the set of reflections $s_{\alpha,n}$ at hyperplanes having a codimension $1$ intersection with the closure of $A_0$. So $\hCS=\{s_{\alpha,0}\mid \text{$\alpha\in R^+$ is simple}\}\cup\{s_{\gamma,1}\}$, where $\gamma\in R^+$ is the highest root. The right action of $\hCS$ on $\CA$ will be of particular importance for us. 
\subsection{A right action of $\hCW$ on $\CA$}

Denote by $A_e\in\CA$ the unique alcove that contains $0\in V$ in its closure and is contained in the dominant Weyl chamber 
$\{\mu\in V\mid \langle \mu,\alpha^\vee\rangle>0\text{ for all $\alpha\in R^+$}\}.
$
Then the map $\hCW\to\CA$, $w\mapsto A_w:=w(A_e)$ is a bijection. The right action of $\hCW$ on $\CA$ is then obtained by transport of structure, i.e.  
 $(A_x)w:=A_{xw}$ for all $x,w\in\hCW$. It commutes with the left action. % In particular, we obtain an action of $\hCW$ on the set $\CV$ of (left) $\DZ R$-orbits in $\CA$. 
 As every connected component of $\CA_T$ is a $\hCW_T$-orbit, the right action of $\hCW$ permutes the set of connected components. 
Let $\hCS\subset\hCW$ be the set of reflections at hyperplanes that have a codimension $1$ intersection with the closure of $A_e$ (these reflections are called the simple affine reflections). 
We are especially interested in the right action of the elements in $\hCS$ on the set $\CA$. 

Let $T$ be a base ring.
For $C\in\CA$ denote by $\lambda_C$ the barycenter in $V$ of $C$.  

\begin{lemma}\label{lemma-specalc} Let $A,B\in\CA$ with $\lgl\lambda_B-\lambda_A,\alpha^\vee\rgl>0$ for all $\alpha\in R^+$. Then there are simple reflections   $s_1$, \dots, $s_n\in\hCS$ such that $A=Bs_1\cdots s_n$ and for all $i=1\dots,n$, the alcoves $Bs_1\cdots s_{i-1}$ and $Bs_1\cdots s_i$ are either $\preceq_T$-incomparable, or $Bs_1\cdots s_{i}\prec_T  Bs_1\cdots s_{i-1}$.
\end{lemma}

\begin{proof} 
Set $\delta=\lambda_B-\lambda_A$. There are  interior points  $x\in B$, $y\in A$ such that $y=x-\delta$ and such that the straight line segment from $x$ to $y$ does not pass through the intersection of at least two hyperplanes. Hence there exists a sequence of alcoves $A=A_0, A_1,\cdots, A_n=B$ such that $A_{i}$ and $A_{i+1}$  share a wall for all $i=0,\dots, n-1$.  Hence $A_i=A_{i-1}s_i$ for some $s_i\in \hCS$, and  $A_{i}\prec_T A_{i-1}$ if the two alcoves are comparable.
\end{proof}

Let us fix $s\in\hCS$ now. Here is another result that  is not true for general base rings. 

\begin{lemma}\label{lemma-Asuparrow} Suppose that $T$ is subgeneric and let $\alpha\in R^+$ be such that $R_T^+=\{\alpha\}$. Let $\Lambda$ be a connected component that satisfies $\Lambda=\Lambda s$. For any  $A\in\Lambda$ we then have $As\in\{\alpha\uparrow A,\alpha\downarrow A\}$.
\end{lemma}
\begin{proof} Note that $A$ and $As$ are not contained in the same $\DZ R$-orbit and $\{A,As\}$ is a $\preceq_T$-interval by Lemma 6.1 in \cite{FieLanWallCross}. The claim follows now from Lemma \ref{lemma-upmin}.
\end{proof}

We say that a subset $\CT$ of $\CA$ is {\em $s$-invariant} if $\CT=\CT s$. For general subsets  $\CT$ define
$$
\CT^\sharp=\CT\cup\CT s\text{ and }\CT^\flat=\CT\cap\CT s.
$$
Then $\CT^\sharp$ is the smallest $s$-invariant set that contains $\CT$, and $\CT^\flat$ is the largest $s$-invariant subset of $\CT$. 

\begin{lemma}{\cite[Lemma 6.2, Lemma 6.3]{FieLanWallCross}}\label{lemma-sharpflat}  Let $\CJ$ be a $T$-open subset of $\CA_T$. 
\begin{enumerate}\item Then  $\CJ^\sharp$ and $\CJ^\flat$ are $T$-open.
\item Suppose that $\CJ$ is contained in a connected component $\Lambda$ of $\CA_T$ that satisfies $\Lambda=\Lambda s$. If  $B\in\CJ^\sharp\setminus\CJ$, then $Bs\preceq_T B$. If $B\in\CJ\setminus\CJ^\flat$, then $B\preceq_TBs$.
\end{enumerate}
\end{lemma}

\subsection{The structure algebra}\label{subsec-strucalg}
Here is the main algebraic ingredient of our theory. 

\begin{definition}\label{def-strucalg}
\begin{enumerate}
\item For a subset $\CL$ of $\CV$ set 
$$
\CZ_S(\CL):=\left\{(z_x)\in\bigoplus_{x\in \CL}S\left|\begin{matrix}\, z_x\equiv z_{s_{\alpha} x}\mod\alpha^\vee \\ \text{ for all $x\in\CL$ and  $\alpha\in R^+$ } \\ \text{ with $s_{\alpha} x\in\CL$}\end{matrix}\right\}\right..
$$ 
%\item For all $\gamma\in\DZ R$ and all $A\in\CA$  we have $z_A= z_{A+\gamma}$. 
\item For a base ring $T$ define $\CZ_T(\CL):=\CZ_S(\CL)\otimes_ST$. 
\end{enumerate}
\end{definition} 
We write $\CZ_S$ and $\CZ_T$ instead of $\CZ_S(\CV)$ and $\CZ_T(\CV)$. Mostly we fix a base ring $T$ and use the even handier notation $\CZ(\CL):=\CZ_T(\CL)$ and $\CZ=\CZ_T$.  If the base ring $T$ is a graded $S$-algebra, then $\CZ$ is a graded $T$-algebra and again we assume, in this case, all $\CZ$-modules to be graded. 
As $T$ is flat as an $S$-module, we obtain a natural inclusion $\CZ_T(\CL)\subset\bigoplus_{x\in\CL} T$. 

Consider  the projection $p^\CL\colon\bigoplus_{x\in\CV} T\to\bigoplus_{x\in\CL} T$ along the decomposition. This induces a homomorphism $\CZ\to\CZ(\CL)$. Set   $\CZ^{\CL}:=p^\CL(\CZ)$. So this is a sub-$T$-algebra of $\CZ(\CL)$. Clearly, for $\CL^\prime\subset\CL$ we have an obvious surjective homomorphism $\CZ^{\CL}\to\CZ^{\CL^\prime}$. 
For a flat homomorphism $T\to T^\prime$  of base rings the natural homomorphism $\bigoplus_{x\in\CL}T\to\bigoplus_{x\in\CL}T^\prime$ induces an isomorphism $\CZ_T^{\CL}\otimes_TT^\prime\cong\CZ_{T^\prime}^\CL$. 

\begin{lemma}{\cite[Lemma 4.1]{FieLanWallCross}}\label{lemma-candec} The direct sum of the homomorphisms $p^{\ol\Lambda}$ induces an isomorphism $\CZ\xrightarrow{\sim}\bigoplus_{\Lambda\in C(\CA_T)}\CZ(\ol\Lambda)$. Hence, $\CZ^{\ol\Lambda}=\CZ(\ol\Lambda)$. \end{lemma}
We call $\CZ=\bigoplus_{\Lambda\in C(\CA_T)}\CZ^{\ol\Lambda}$ the {\em canonical decomposition} of $\CZ$.

\subsection{$\CZ$-modules} 
For $\alpha\in R^+$ define $T^{\alpha}:=T[\beta^{\vee-1}\mid\beta\in R^+,\beta\ne\alpha] $
and $T^\emptyset:=T[\beta^{\vee-1}\mid \beta\in R^+]$.  

\begin{definition}  
A $\CZ$-module $M$ is called 
\begin{enumerate}
\item {\em root torsion free} if the left multiplication with $\alpha^\vee$ on $M$ is injective for all $\alpha\in R^+$, i.e. if the canonical map $M\to M\otimes_TT^{\emptyset}$ is injective, and
\item  {\em root reflexive} if $M$ is root torsion free and $M=\bigcap_{\alpha\in R^+}M\otimes_TT^{\alpha}$ as subsets in $M\otimes_TT^{\emptyset}$. 
\end{enumerate}
\end{definition}  

Let  $M$ be a root torsion free $\CZ$-module. For a subset $\CL$ of $\CV$ we define $M^{\CL}:=\CZ^{\CL}\otimes_{\CZ} M$ (the definition in \cite{FieLanWallCross} of $M^{\CL}$ was a little more technical, but see Lemma 4.4 in loc.cit.). This is a root torsion free $\CZ$-module again by Lemma 4.5 in \cite{FieLanWallCross}, and we denote by $p^{\CL}\colon M\to M^{\CL}$ the homomorphism $m\mapsto 1\otimes m$. Let $T\to T^\prime$ be a flat homomorphism of base rings. Let $M$ be a $\CZ_T$-module. Then $M\otimes_TT^{\prime}$ is a $\CZ_{T^\prime}$-module and we have a canonical isomorphism  $M^\CL\otimes_TT^\prime\cong(M\otimes_TT^{\prime})^{\CL}$.  In the case $\CL=\{x\}$ we write $M^x$ instead of $M^{\{x\}}$ and call this the {\em stalk} of $M$ at $x$. 
\begin{definition} We say that $M$ is {\em $\CZ$-supported on $\CL$} if $p^\CL\colon M\to M^{\CL}$ is an isomorphism. \end{definition}

\begin{remark}\label{rem-Zsupp} If $f\colon M\to N$ is a homomorphism between $\CZ$-modules and $N$ is $\CZ$-supported on $\CL$, then $f$ factors uniquely over $p^{\CL}\colon M\to M^{\CL}$. 
\end{remark}

\subsection{$s$-invariants in $\CZ$}\label{sec-sinvZ}
Let $s\in\hCS$. Note that the right action of $\hCW$ on $\CA$ preserves (left) $\DZ R$-orbits, hence it induces a right action of $\hCW$ on $\CV$. Again we call a subset $\CL$ of $\CV$ {\em $s$-invariant} if $\CL=\CL s$. 
Let $\CL\subset\CV$ be  $s$-invariant. 
Denote by  $\eta_s\colon\bigoplus_{x\in\CL}T\to\bigoplus_{x\in\CL}T$  the homomorphism given by $\eta_s(z_x)=(z^\prime_x)$, where $z^\prime_x=z_{xs}$ for all $x\in\CL$. This preserves the subalgebra  $\CZ^\CL$ (cf. \cite[Lemma 6.6]{FieLanWallCross}). We denote by $\CZ^{\CL,s}$ the subalgebra of $\eta_s$-invariant elements, and by $\CZ^{\CL,{-s}}$ the sub-$\CZ^{\CL,s}$-module of $\eta_s$-antiinvariant elements. %We set $\CZ(\CL)^s=\CZ_S(\CL)^s\otimes_S T$ and $\CZ(\CL)^{-s}=\CZ_S(\CL)^{-s}\otimes_S T$. 

\begin{lemma} \label{lemma-transsinv} \cite[Lemma 6.8,  Lemma 6.10]{FieLanWallCross} Let $\CL\subset\CV$ be $s$-invariant.
\begin{enumerate}
\item $\CZ^{\CL,-s}\cong\CZ^{\CL, s}[-2]$ as $\CZ^s$-modules. Hence $\CZ^{\CL}\cong\CZ^{\CL,s}\oplus\CZ^{\CL,s}[-2]$ as a $\CZ^{\CL,s}$-module. 
\item  For any root torsion free $\CZ$-module  $M$  that is $\CZ$-supported on $\CL$ there is a canonical isomorphism
$$
\CZ\otimes_{\CZ^s}M\cong\CZ^{\CL}\otimes_{\CZ^{\CL,s}} M.
$$
In particular, $\CZ\otimes_{\CZ^s}M$ is again $\CZ$-supported on $\CL$.
\end{enumerate}
\end{lemma} 
Note that the grading shift in the above theorem only makes sense in the case that $T$, and hence $\CZ$, is a graded $S$-algebra. In the non-graded setup, the statements of the lemma hold with the shifts omitted.

\begin{lemma} \label{lemma-decZ}\cite[Lemma 6.11]{FieLanWallCross} Suppose that $\Lambda\in C(\CA_T)$ satisfies $\Lambda\ne \Lambda s$. Then the composition $\CZ^{\ol{\Lambda\cup\Lambda s},s}\subset\CZ^{\ol{\Lambda\cup\Lambda s}}=\CZ^{\ol\Lambda}\oplus\CZ^{{\ol{\Lambda s}}}\xrightarrow{pr_\Lambda}\CZ^{\ol\Lambda}$ is an isomorphism of $T$-algebras. % In particular, if $M$ is a $\CZ$-module supported on $\Lambda$ we obtain a decomposition
%$$
%(\CZ\otimes_{\CZ^s} M)^{\CJ}\cong M^{\CJ}\oplus M[\Lambda s]^{(s)}
%$$
%of $\CZ$-modules that is functorial in $M$. 
\end{lemma}
\subsection{Graded free quotients of the structure algebra} \label{subsec-quotZ} For a graded free $S$-module $M\cong\bigoplus_{n\in\DZ} S[n]^{r_n}$ of finite rank we set $p(M)=\sum_{n\in\DZ} r_nv^{-n}\in\DZ[v, v^{-1}]$. This is  the {\em graded rank} of $M$.
Let $\mu\in X$ and recall the definition of the $(T,\mu)$-topology on $\CV$  in Section \ref{subsec-toponV}. Recall that we denote the set $\CV$ endowed with the $(T,\mu)$-topology by $\CV_{T,\mu}$.  

\begin{proposition}\label{prop-Zlocfree}
 Let $\CY$ be $(T,\mu)$-open in $\CV$.
 \begin{enumerate}
\item The inclusion $\CZ_S^\CY\subset\CZ_S(\CY)$ is an isomorphism after applying the functor $(\cdot)\otimes_ST$. 
%$$
%\CZ^{\CY}_T=\left\{ (z_x)\in\bigoplus_{x\in\CY} S\left|\begin{matrix}z_x\equiv z_{s_\alpha(x)}\mod\alpha^\vee\\ \text{ for $x\in\CY$, $\alpha\in R^+$}\\ \text{ with $s_\alpha(x)\in\CY$}\end{matrix}\right\}\right.\otimes_ST.
%$$
\item  Suppose that $\CY$ is contained in a connected component of $\CV_{T,\mu}$. Then $\CZ_S(\CY)$ is a graded free $S$-module of graded rank $p(\CZ_S({\CY}))=\sum_{x\in\CY} v^{2n_{T,x}}$, where  $n_{T,x}$ is the number of $\alpha\in R_T^+$ with $s_{\alpha}(x)\le_{T,\mu} x$. 
\end{enumerate}
\end{proposition}
\begin{proof} Let us prove (1).
As $\CV_{T,\mu}$ splits into the connected components $\ol{\Lambda_\mu}$ and as $\CZ_S\otimes_ST=\bigoplus_{\Lambda\in C(\CA_T)}\CZ_S({\ol\Lambda})\otimes_ST$ we can assume that $\CY\subset\ol{\Lambda_\mu}$ for some $\Lambda\in C(\CA_T)$. Let 
$A\in\Lambda_\mu$ be minimal. By Lemma \ref{lemma-tauiso} the map $\tau\colon\CW_{T,\mu}\to\Lambda_\mu$, $w\mapsto w(A)$ is an isomorphism of partially ordered sets. 
Hence $\CX=\tau^{-1}(\pi^{-1}(\CY)\cap\Lambda_\mu)$ is an order ideal in $\CW_{T,\mu}$ with respect to the Bruhat order $\le_{T,\mu}$.  Now $\CZ_S(\ol{\Lambda_\mu})$  is the structure algebra of the moment graph $\CG_{T}$ associated with the root system $R_T^+$, and  $\CZ_S^{\CY}$ is hence its  image under the projection $\bigoplus_{x\in\CW_{T,\mu}}S\to\bigoplus_{x\in\CX} S$ along the decomposition. 
By \cite[Theorem 6.1, Corollary 6.1]{LanJoA} the Braden--MacPherson sheaf on the Bruhat graph associated with the root system $R_T^+$ has stalk multiplicities $1$  (for all $k$ satisfying the GKM-condition). Hence it coincides with the structure sheaf, so $\CZ_S(\ol{\Lambda_\mu})$ and $\CZ_S(\CX)$ can be interpreted as the global and the local sections of the  Braden--MacPherson sheaf on $\CG_T$ (note that here we can refer to local sections as $\CX$ is an order ideal). The Braden--MacPherson sheaf is flabby. In particular, together with the identifications above this implies that $\CZ_S(\ol{\Lambda_\mu})\to\CZ_S(\CY)$ is surjective. Lemma \ref{lemma-candec} implies that $\CZ_S\to\CZ_S(\ol{\Lambda_\mu})=\CZ_S(\ol\Lambda)$ is surjective after applying the functor $(\cdot)\otimes_S T$. The last two statements imply that $\CZ_S\to\CZ_S(\CY)$ is surjective after  applying $(\cdot)\otimes_ST$, hence (1). 

Now we prove (2). Assume that $\ol\Lambda_\mu\subset\CV_{T,\mu}$ is the connected component containing $\CY$. As before we translate the statement (2) into a statement involving the moment graph $\CG_T$ associated with $\ol\Lambda_\mu$. By \cite[Corollary 6.5]{FieTAMS}, the Braden--MacPherson sheaves on the moment graph $\CG_T$ admit Verma flags, i.e. the local sections are graded free $S$-modules of finite rank. Using the mentioned  identification we deduce  that $\CZ_S(\CY)$ is a graded free $S$-module of finite rank. 
It remains to prove the statement about the multiplicities. We prove this by induction on the number of elements in $\CY$. If $\CY$ is empty, then $\CZ_S(\CY)=\{0\}$ and $p(\CZ_S(\CY))=0$ in accordance with the alleged equation. Suppose $\CY$ is non-empty. Then there is an  element $w$ in $\CY$ such that $\CY^{\prime}=\CY\setminus\{w\}$ is $(T,\mu)$-open again (because the topology on $\CV$ is induced by a partial order on $\CK_\mu$). By what we have shown above, the kernel of the (surjective) homomorphism  $\CZ_S(\CY)\to\CZ_S(\CY^\prime)$ can be identified with the sections of the structure algebra of $\CG_T$ that vanish when restricted to $\CX^{\prime}=\tau^{-1}(\pi^{-1}(\CY^\prime)\cap\ol{\Lambda_\mu})$, i.e. with the graded $S$-module $\alpha_1^\vee\cdots\alpha^\vee_n S$, where $\{\alpha_1,\dots,\alpha_n\}$ is the subset in $R_T^+$ containing all $\alpha$ with the property $s_\alpha(w)\le_{T,\mu} w$.  Hence  $p(\CZ_S(\CY))=p(\CZ_S(\CY^\prime))+v^{2n}$ and the claim  follows by induction. \end{proof}

\section{(Pre-)Sheaves on $\CA_T$}
In this section we review the construction of the categories $\bS_T$ and $\bP_T$ of (pre-) sheaves of $\CZ_T$-modules on $\CA_T$. Again, the main reference is \cite{FieLanWallCross}. We add some new results on morphisms in $\bP_T$. Then we study the exact structure on $\bS_T$ that is inherited by its inclusion into the abelian category of sheaves of $\CZ_T$-modules on $\CA_T$.

\subsection{Preliminaries}
Fix a base ring $T$. Let $\SM$ be a presheaf of $\CZ=\CZ_T$-modules on the topological space $\CA_T$.  For any open subset $\CJ$ of $\CA_T$ we write $\SM{(\CJ)}$ for the $\CZ$-module of sections of $\SM$ over $\CJ$. We often write  $\Gamma(\SM)$ for the space $\SM(\CA)$ of global sections of $\SM$.
 For an inclusion $\CJ^\prime\subset\CJ$  denote by $r_{\CJ}^{\CJ^\prime}\colon \SM(\CJ)\to\SM{(\CJ^\prime)}$ the restriction homomorphism. Most often  we write $m|_{\CJ^\prime}$ instead of $r_\CJ^{\CJ^\prime}(m)$. If the base ring $T$ is graded, so is the $T$-algebra $\CZ$, and then we assume that all local sections  are graded. If $\SM$ is such a presheaf, we obtain, for $n\in\DZ$, the shifted presheaf $\SM[n]$ by applying the shift $(\cdot)[n]$ to all local sections.

\begin{definition}
Let $\SM$ be a presheaf of $\CZ$-modules on $\CA_T$. \begin{enumerate}\item  We say that $\SM$ is {\em finitary} if $\SM(\emptyset)=0$ and if there exists a finite subset $\CT$ of $\CA$ such that for any $T$-open subset $\CJ$  the restriction  $\SM(\CJ)\to\SM((\CJ\cap \CT)_{\preceq_T})$ is an isomorphism. \item We say that $\SM$ is {\em root torsion free} if any $\CZ$-module of local sections is root torsion free. \item We say that $\SM$ is {\em root reflexive} if any $\CZ$-module of local sections is root  reflexive.\end{enumerate}  
\end{definition}

\subsection{The support condition}
Let $\SM$ be a  root torsion free presheaf.  If  $\CL$ is a subset of $\CV$, then we can define a new presheaf $\SM^{\CL}$ by composing with the functor $(\cdot)^\CL$. 
 Again we write $\SM^x$ instead of $\SM^{\{x\}}$. Denote by $i_x\colon x\to\CA$ the inclusion.  

\begin{definition} 
We say that $\SM$ satisfies the {\em support condition} if it is root torsion free and if the natural morphism $\SM^x\to i_{x\ast}i_x^\ast\SM^x$ is an isomorphism of presheaves for all $x\in\CV$. 
\end{definition} 

Explicitely, this means that for any $x\in\CV$ and any $T$-open subset $\CJ$, the homomorphism $\SM(\CJ)^x\to\SM((\CJ\cap x)_{\preceq_T})^x$, i.e. the $x$-stalk 
of the restriction homomorphism, is an isomorphism. The following may justify the notion of ``support condition''.

\begin{lemma}\label{lemma-suppcond} \cite[Lemma 5.2]{FieLanWallCross} 
Let  $\SM$ be a root torsion free and flabby presheaf. Then $\SM$ satisfies the support condition if and only if the following holds.  For any inclusion  $\CJ^\prime\subset\CJ$ of $T$-open sets the kernel of the restriction homomorphism $\SM(\CJ)\to\SM{(\CJ^\prime)}$ is $\CZ$-supported on $\pi(\CJ\setminus\CJ^\prime)$. 
\end{lemma}

\begin{definition}
Denote by $\bP=\bP_T$ the full subcategory of the category of presheaves of $\CZ$-modules on $\CA_T$ that contains all objects $\SM$ that are flabby, finitary,  root torsion free and satisfy the support condition. 
\end{definition}

Let $\SM$ be a presheaf of $\CZ$-modules on $\CA_T$. The canonical decomposition $\CZ=\bigoplus_{\Lambda\in C(\CA_T)}\CZ^{\ol\Lambda}$ induces a direct sum decomposition $\SM=\bigoplus_{\Lambda\in C(\CA_T)}\SM^{\ol\Lambda}$ of presheaves on $\CA_T$.

\begin{lemma}[{\cite[Lemma 5.4]{FieLanWallCross}}] \label{lemma-supdirsum} 
If $\SM$ is an object in $\bP$, then $\SM^{\ol\Lambda}$ is supported, as a presheaf, on $\Lambda$ for all $\Lambda\in C(\CA_T)$.
\end{lemma}

Recall  that $\SM$ is {\em supported as a presheaf} on $\CT\subset\CA$ if the restriction homomorphism $\SM(\CJ)\to\SM((\CJ\cap\CT)_{\preceq_T})$ is an isomorphism for all $T$-open subsets $\CJ$.
The following two results did not appear in \cite{FieLanWallCross}. 

\begin{lemma}\label{lemma-supp2}  
Let  $\SM$ be an object in $\bP$ that is supported as a presheaf on the locally closed subset  $\CK$ of $\CA_T$. Then  the $\CZ$-module $\SM(\CJ)$ is $\CZ$-supported on $\pi(\CJ\cap\CK)$ for any $T$-open subset $\CJ$.
\end{lemma}

\begin{proof} 
Set $\CO:=\CK_{\preceq_T}\setminus\CK$. This is a $T$-open subset of $\CA$ by Lemma \ref{lemma-top}.  Consider the restriction homomorphisms  $\SM(\CJ)\to\SM(\CJ\cap\CK_{\preceq_T})\to\SM(\CJ\cap\CO)$.  As $\CJ\cap\CK_{\preceq_T}\cap\CK=\CJ\cap\CK$  the first homomorphism is an isomorphism. As $\CJ\cap\CO\cap\CK=\emptyset$ and as $\SM$ is finitary, $\SM(\CJ\cap\CO)=\SM(\emptyset)=0$. By Lemma \ref{lemma-suppcond}, $\SM(\CJ)$ is $\CZ$-supported on $\pi((\CJ\cap\CK_{\preceq_T})\setminus(\CJ\cap\CO))=\pi(\CJ\cap\CK)$.  
\end{proof}

\begin{lemma} \label{lemma-homs2}  
Suppose $\CK\subset\CA_T$ is a section. Let $\SM$ be an object in $\bP$ supported on $\CK$, and let $\SN$ be an object in $\bP$ supported on $\CK_{\succeq_T}$. 
\begin{enumerate}
\item For any inclusion $\CJ^\prime\subset \CJ$ of $T$-open sets the restriction homomorphism $\SM(\CJ)\to\SM(\CJ^{\prime})$ induces an isomorphism $\SM(\CJ)^{\pi(\CJ^\prime\cap\CK)}\xrightarrow{\sim}\SM(\CJ^\prime)$.
\item The functor of global sections induces an isomorphism $\Hom_{\bP}(\SM,\SN)\xrightarrow{\sim}\Hom_{\CZ}(\Gamma(\SM),\Gamma(\SN))$.
\end{enumerate}
\end{lemma}

\begin{proof} 
(1) From Lemma \ref{lemma-supp2} and Remark \ref{rem-Zsupp} we obtain a homomorphism $\SM(\CJ)^{\pi(\CJ^\prime\cap\CK)}\to\SM(\CJ^\prime)$. As $\SM$ is flabby, this homomorphism is surjective.  We now show that the kernel of $\SM(\CJ)\to\SM(\CJ^\prime)$  is $\CZ$-supported on the complement of  $\pi(\CJ^\prime\setminus\CK)$ in $\CV$, which yields the injectivity of the former homomorphism. As $\SM$ is supported on $\CK$, the above homomorphism identifies with the restriction homomorphism  $\SM((\CJ\cap\CK)_{\preceq_T})\to\SM((\CJ^\prime\cap\CK)_{\preceq_T})$. By Lemma \ref{lemma-suppcond}, its kernel   is hence $\CZ$-supported on $\pi((\CJ\cap\CK)_{\preceq_T}\setminus(\CJ^\prime\cap\CK)_{\preceq_T})$, so we need to show that $\pi((\CJ\cap\CK)_{\preceq_T}\setminus(\CJ^\prime\cap\CK)_{\preceq_T})\cap\pi(\CJ^\prime\cap\CK)=\emptyset$. Let $A\in\CJ^\prime\cap\CK$ and suppose that $A+\tau\in(\CJ\cap\CK)_{\preceq_T}$. Then $A+\tau\in\CK_{\preceq_T}$. As $\CK$ is a section, we have  $\CK_{\preceq_T}=\bigcup_{\gamma\ge 0}\CK-\gamma$  and $\CK+\gamma\cap\CK+{\gamma^\prime}=\emptyset$ for $\gamma\ne\gamma^\prime$. So $\tau\le 0$ and hence $A+\tau\in(\CJ^\prime\cap\CK)_{\preceq_T}$, which proves our claim.

(2) The injectivity of the homomorphism induced by the functor of global sections follows from the fact that $\SM$ is  flabby. It remains to prove that the homomorphism is surjective. So let $g\colon \Gamma(\SM)\to\Gamma(\SN)$ be a homomorphism of $\CZ$-modules. Let $\CJ$ be open in $\CA_T$. By Lemma \ref{lemma-supp2},  $\SN(\CJ)$ is $\CZ$-supported on $\pi(\CJ\cap\CK_{\succeq_T})$. We claim that this coincides with $\pi(\CJ\cap\CK)$. Clearly $\pi(\CJ\cap\CK)\subset\pi(\CJ\cap\CK_{\succeq_T})$. The converse inclusion follows from the defining property  $\CK_{\succeq_T}=\bigcup_{\gamma\ge0}\CK+\gamma$ of a section and the fact that $A+\gamma\in\CJ$ implies $A\in\CJ$ for $\gamma\ge 0$. 

Hence the composition   $\Gamma(\SM)\xrightarrow{g}\Gamma(\SN)\to\SN(\CJ)$ factors over $\Gamma(\SM)\to \Gamma(\SM)^{\pi(\CJ\cap\CK)}$, and so, using (1), induces a homomorphism $f(\CJ)\colon \SM(\CJ)\xleftarrow{\sim}\Gamma(\SM)^{\pi(\CJ\cap\CK)}\to \SN(\CJ)$. So we constructed a homomorphism $\SM(\CJ)\to\SN(\CJ)$ for any open subset $\CJ$. Clearly these are compatible with restrictions, and hence yield a morphism $f\colon\SM\to\SN$ of presheaves with $\Gamma(f)=g$.
\end{proof}

\subsection{Admissible families and rigidity}
Suppose that $\DT$ is a family of subsets in $\CA$. 

\begin{definition}\label{def-adm} 
We say that $\DT$ is a  {\em $T$-admissible family}, if it  satisfies the following.
\begin{itemize}
\item $\CA\in\DT$.
\item Each element in $\DT$ is $T$-open. \item $\DT$ is stable under finite intersections. 
\item For any $T$-open subset $\CJ$ and any $x\in\CV$, there is some $\CJ^\prime\in\DT$ with $\CJ\cap x=\CJ^\prime\cap x$.
\end{itemize}
\end{definition}

Here are some examples of admissible families. 

\begin{lemma} \cite[Lemma 3.15, Lemma 6.4 and Lemma 6.5]{FieLanWallCross} \label{lem-admfam}
\begin{enumerate}
\item For $s\in\hCS$,  the set of $s$-invariant $T$-open subsets of $\CA$ is $T$-admissible.
\item Let $T\to T^\prime$ be a homomorphism of base rings. Then the set of $T$-open subsets in $\CA$ is $T^\prime$-admissible.
\item Let $T\to T^\prime$ be a homomorphism of base rings and $s\in\hCS$. Then the set of $s$-invariant $T$-open subsets in $\CA$ is $T^\prime$-admissible.
\end{enumerate}
\end{lemma}

The following result is the reason why the above is a useful  notion. 
\begin{proposition}\label{prop-rig} \cite[Proposition 5.7]{FieLanWallCross} Suppose that $\DT$ is a $T$-admissible family. Let  $\SM$ and $\SN$ be objects in $\bP$.  Suppose we are given  a homomorphism $f^{(\CJ)}\colon \SM {(\CJ)}\to \SN {(\CJ)}$ of $\CZ$-modules for each $\CJ\in\DT$ in such a way that for any inclusion $\CJ^\prime\subset\CJ$ of sets in $\DT$   the diagram

\centerline{
\xymatrix{
\SM {(\CJ)}\ar[rr]^{f^{(\CJ)}}\ar[d]_{r_\CJ^{\CJ^\prime}}&&\SN {(\CJ)}\ar[d]^{r_{\CJ}^{\CJ^\prime}}\\
\SM {(\CJ^\prime)}\ar[rr]^{f^{(\CJ^\prime)}}&&\SN {(\CJ^\prime)}
}
}
\noindent 
commutes. Then there is a unique morphism $f\colon \SM \to \SN $ in $\bP$  such that $f{(\CJ)}=f^{(\CJ)}$ for all $\CJ\in\DT$. 
\end{proposition}

\subsection{Base change}
Suppose that $T\to  T^\prime$ is a  flat homomorphism of  base rings. The following statement defines and characterizes  a base change functor $\cdot\boxtimes_TT^\prime$ from $\bP_T$ to $\bP_{T^\prime}$. 

\begin{proposition} \cite[Proposition 5.8]{FieLanWallCross} \label{prop-defbox}
There is a unique functor $\cdot\boxtimes_TT^\prime\colon \bP_T\to\bP_{T^\prime}$ with the following property. For any $T$-open subset $\CJ$ in $\CA$ there is an isomorphism $(\SM\boxtimes_TT^\prime)(\CJ)\cong\SM(\CJ)\otimes_TT^\prime$ of $\CZ_{T^\prime}$-modules that is functorial in $\SM$ and compatible with the restriction homomorphisms for an inclusion of $T$-open subsets. 
\end{proposition} 

The following defines the most relevant category for our purposes.

\begin{definition}\label{def-C} 
Denote by $\bS=\bS_T$ the full subcategory of the category $\bP_T$ that contains all objects $\SM$ that have the property that for all flat homomorphisms $T\to T^\prime$ of base rings the presheaf $\SM\boxtimes_TT^\prime$ is a root reflexive sheaf on $\CA_{T^\prime}$. 
\end{definition}
The most important aspect of the above definition is that $\SM\boxtimes_TT^\prime$ is supposed to be a sheaf on $\CA_{T^\prime}$, not  just a presheaf.
\subsection{An exact structure on $\bS$ } 
The category $\bS$ is  not abelian. One reason for this is that the objects in $\bS$  are supposed to be torsion free, which prohibits the existence of general cokernels. However, the category $\bS$ is a full subcategory of a category of sheaves of $\CZ$-modules on $\CA_T$, hence it inherits an exact structure. Here we collect some first properties of this exact structure. 

\begin{lemma}\label{lemma-ses1}
Let $0\to\SA\to\SB\to\SC\to 0$ be a sequence in $\bS$. Then the following statements are equivalent:
\begin{enumerate}
\item The sequence is exact (i.e. exact in the category of sheaves of $\CZ$-modules on $\CA_T$).
\item For any open subset $\CJ$ of $\CA_T$, the sequence
$
0\to\SA(\CJ)\to\SB(\CJ)\to\SC(\CJ)\to 0
$
is exact in the category of $\CZ$-modules (i.e. the sequence is exact as a sequence of presheaves).
\item For any $A\in\CA_T$, the sequence 
$
0\to\SA({\preceq_T A})\to\SB({\preceq_T A})\to\SC({\preceq_T A})\to 0
$
is exact  in the category of $\CZ$-modules.
\end{enumerate}
\end{lemma}
%Note that $\{\preceq_TA\}$ is the smallest $T$-open set that contains $A$, hence $\SM(\preceq_T A)$ is the stalk of the (pre-)sheaf $\SM$ on $A$. 

\begin{proof} 
Recall that a sequence of sheaves is exact if and only if the induced sequences on all stalks are exact. Hence (1) and (3) are equivalent. Moreover, (3) is a special case of (2).  And (1) implies (2) by standard arguments (using the flabbiness of $\SA$). 
\end{proof}

Note that, in particular, presheaf-(co-)kernels and sheaf-(co-)kernels for morphisms in the category $\bS$ coincide.

%The category $\bZ_{[\CK]}$ is a full subcategory of the category of $\CZ$-modules. However, we won't use the analogously inherited exact structure. Instead we will define one that depends on the section $\CK$.

\section{Restriction functors and subquotients}
In this section we consider two constructions on objects in $\bS$  that did not appear in \cite{FieLanWallCross}. We define the restriction functor $(\cdot)|_{\CO}$ for an open subset $\CO$  (more precisely, the functor $i_\ast i^\ast$ for the inclusion $i\colon\CO\to\CA_T$), and a subquotient functor $(\cdot)_{[\CO\setminus\CO^\prime]}$ associated with an inclusion $\CO^\prime\subset\CO$ of open subsets.  

\subsection{The restriction functor} 
Fix a base ring $T$. Let $\CO\subset\CA_T$ be an open subset and denote by $i\colon\CO\to\CA_T$ the inclusion. Then $\CO$ inherits the structure of a topological space.  Given a  presheaf $\SM$ on $\CA_T$ we define the presheaf $\SM|_\CO:=i_\ast i^\ast\SM$ on $\CA_T$. More explicitely, for an open subset $\CJ$ the local sections are 
$$
\SM|_\CO(\CJ)=\SM({\CJ\cap\CO}),
$$
and the  restriction homomorphism is $r_{\CJ\cap\CO}^{\CJ^\prime\cap\CO}$ for $\CJ^\prime\subset\CJ$. For $\CO^\prime\subset\CO$ there is an obvious morphism $\SM|_{\CO}\to\SM|_{\CO^\prime}$ of presheaves. Note that if $\SM$ is a sheaf, then $\SM|_\CO$ is also a sheaf.

\begin{lemma}\label{lemma-admres} 
Let $\CO$ be an open subset of $\CA_T$.
\begin{enumerate}
\item If $\SM$ is a flabby presheaf satisfying the support condition, then $\SM|_{\CO}$ is flabby and  satisfies the support condition as well. \item If $\SM$ is an object in $\bP$, then so is $\SM|_\CO$.
\item For a flat homomorphism $T\to T^\prime$ of base rings and an object $\SM$ of $\bP$ we have $(\SM\boxtimes_TT^\prime)|_\CO=(\SM|_\CO)\boxtimes_TT^\prime$ functorially in $\SM$.
\item  If $\SM$ is an object in $\bS$, then so is $\SM|_\CO$.
\end{enumerate} 
\end{lemma}

\begin{proof} 
(1) Clearly $\SM|_\CO$ is a flabby presheaf. Hence we can use the statement in Lemma \ref{lemma-suppcond}.  So let $\CJ^\prime\subset \CJ$ be open in $\CA_T$. We need to show that the kernel of $\SM(\CJ\cap\CO)\to\SM(\CJ^\prime\cap\CO)$ is $\CZ$-supported on $\pi(\CJ\setminus\CJ^\prime)$. But as $\SM$ satisfies the support condition, this kernel is $\CZ$-supported even on the subset $\pi((\CJ\cap\CO)\setminus(\CJ^\prime\cap\CO))$ of $\pi(\CJ\setminus\CJ^\prime)$.

(2) Let $\SM$ be an object in $\bP$.  It is clear that $\SM|_{\CO}$ is finitary, flabby and root torsion free. By (1) it also satisfies the support condition.

(3) Let $\CJ$ be a $T$-open. Then $(\SM|_{\CO})\boxtimes_TT^\prime(\CJ)=\SM|_{\CO}(\CJ)\otimes_TT^\prime=\SM(\CJ\cap\CO)\otimes_TT^\prime=(\SM\boxtimes_TT^\prime)(\CJ\cap\CO)=(\SM\boxtimes_TT^\prime)|_\CO(\CJ)$ functorially in $\SM$. As $\SM|_\CO\boxtimes_TT^\prime$ and $(\SM\boxtimes_TT^\prime)|_\CO$ are objects in $\bP$ by the above, the statement (3) follows from Proposition \ref{prop-rig} and the fact that the $T$-open sets form a $T^\prime$-admissible family by Lemma \ref{lem-admfam}. 

(4) It is clear that $\SM|_\CO$ is root reflexive if $\SM$ is, and that it is a sheaf if $\SM$ is a sheaf. The claim follows from this and the above statements. 
\end{proof} 

\subsection{Subquotients}
Let $\SM$ be a presheaf of $\CZ$-modules on $\CA_T$.  Recall that for a locally closed subset $\CK$, the set $\CK_{\preceq_T}\setminus\CK$ is open  by Lemma \ref{lemma-top}.

\begin{definition} 
Define  $\SM_{[\CK]}$ as the kernel (in the category of presheaves) of the natural  morphism $\SM|_{\CK_{\preceq_T}}\to\SM|_{\CK_{\preceq_T}\setminus\CK}$. 
\end{definition}

More explicitely, $\SM_{[\CK]}(\CJ)$ is the kernel of the restriction morphism $\SM(\CJ\cap\CK_{\preceq_T})\to\SM(\CJ\cap(\CK_{\preceq_T}\setminus\CK))$ for any open set $\CJ$, and the restriction morphism is induced by the restriction morphism $\SM(\CJ\cap\CK_{\preceq_T})\to\SM(\CJ^\prime\cap\CK_{\preceq_T})$ for a pair $\CJ^\prime\subset\CJ$.

\begin{remark}\label{rem-suppcond2}
\begin{enumerate}
\item Suppose that $\SM$ is a finitary sheaf. Then $\SM$ is supported on $\CK$ if and only if $\SM=\SM_{[\CK]}$ (easy exercise). In particular, $\SM_{[\CK]}$ is supported on $\CK$. \item If $\SM$ is a flabby presheaf and satisfies the support condition,  then it follows from Lemma \ref{lemma-suppcond} that each local section of $\SM_{[\CK]}$ is $\CZ$-supported on $\pi(\CK)\subset\CV$. 
\end{enumerate}
\end{remark}
The following results are  easy to prove properties of general sheaves on general topological spaces. 

\begin{lemma} \label{lemma-reshoms} 
Let $\CK$ be a locally closed subset of $\CA_T$. Let $\SM$ be an object in $\bS$. 
\begin{enumerate}
\item Suppose $\CO^\prime\subset\CO\subset\CA_T$ are open sets with $\CK=\CO\setminus\CO^\prime$. Then the morphism $\SM|_{\CO}\to\SM|_{\CK_{\preceq_T}}$ induces an isomorphism 
$$
\ker(\SM|_{\CO}\to \SM|_{\CO^\prime})\cong \SM_{[\CK]}.
$$
\item There is a cofiltration  $\SM_{[\CK]}=\SM_0\supset \SM_1\supset\dots\supset \SM_{n+1}=0$ in the category of sheaves with $\SM_{i-1}/\SM_{i}\cong\SM_{[A_i]}$ for some $A_i\in\CK$ such that $A_j\preceq  A_i$ implies $j\le i$.
\end{enumerate}
\end{lemma}

\begin{lemma}\label{lemma-propres} 
Let $\SM$ be an object in $\bP$   and let $\CK\subset\CA_T$ be locally closed. 
\begin{enumerate}
\item  $\SM_{[\CK]}$ is finitary and root torsion free. If $\SM$ is root reflexive, then so is $\SM_{[\CK]}$.  \item If $\SM$ is a sheaf, then $\SM_{[\CK]}$ is a flabby sheaf.
%\item If $\SM$ is a  sheaf satisfying the support condition, then $\SM_{[\CK]}$ satisfies the support condition.
\item For a flat homomorphism $T\to T^\prime$ we have $(\SM\boxtimes_TT^\prime)_{[\CK]}\cong\SM_{[\CK]}\boxtimes_TT^\prime$. 
\item If $\SM$ is an object in $\bS$, then so is $\SM_{[\CK]}$. 
\end{enumerate}
\end{lemma}

\begin{proof}   
Set $\CO=\CK_{\preceq_T}\setminus\CK$. %O\setminus\CO^\prime$ for open subsets $\CO^\prime\subset\CO$. Then $\SM_{[\CK]}$ is the presheaf kernel of $\SM|_{\CO}\to\SM|_{\CO^\prime}$.

(1) It is clear that $\SM_{[\CK]}$ is finitary.  For any open subset $\CJ$, $\SM_{[\CK]}(\CJ)$ is the kernel of $\SM(\CJ\cap\CK_{\preceq_T})\to\SM(\CJ\cap\CO)$. As a submodule of a root torsion free module, $\SM_{[\CK]}(\CJ)$ is root torsion free, and as the kernel of a homomorphism between root reflexive modules, it is root reflexive.

(2) If $\SM$ is a sheaf, then so are $\SM|_{\CK_{\preceq_T}}$ and $\SM|_{\CO}$. As a kernel of a homomorphism of sheaves, $\SM_{[\CK]}$ is a  sheaf (note that even the presheaf kernel is a sheaf). Now let $\CJ$ be open in $\CA_T$, and let $m\in\SM_{[\CK]}(\CJ)$ be a section. Then $m$ is  a section in $\SM(\CJ\cap\CK_{\preceq_T})$ with the property that $m|_{\CJ\cap\CO}=0$. As $\SM$ is a sheaf, there is a section $m^\prime\in\SM((\CJ\cap\CK_{\preceq_T})\cup\CO)$ with $m^\prime|_{\CJ\cap\CK_{\preceq_T}}=m$ and $m^\prime|_{\CO}=0$. As $\SM$ is flabby, there is a section $m^{\prime\prime}\in\SM(\CK_{\preceq_T})$ with $m^{\prime\prime}|_{(\CJ\cap\CK_{\preceq_T})\cup\CO}=m^\prime$. Then $m^{\prime\prime}$ is a global section of $\SM_{[\CK]}$ and is a preimage of $m$. Hence $\SM_{[\CK]}$ is flabby.

(3) Let $\CJ$ be a $T$-open subset. By definition, $\SM_{[\CK]}(\CJ)$ is the kernel of $a\colon\SM(\CJ\cap \CK_{\preceq_T})\to\SM(\CJ\cap\CO)$, and $(\SM\boxtimes_TT^\prime)_{[\CK]}(\CJ)$ is the kernel of $b\colon\SM\boxtimes_TT^\prime(\CJ\cap \CK_{\preceq_T})\to\SM\boxtimes_TT^\prime(\CJ\cap\CO)$ (by Lemma \ref{lemma-reshoms}). By the characterizing property of the base change functor, $b=a\otimes_TT^\prime$.  As $T\to T^\prime$ is flat, we obtain an identification $\ker b=(\ker a)\otimes_TT^\prime$, i.e. $(\SM_{[\CK]}\boxtimes_TT^\prime)(\CJ)=\SM_{[\CK]}(\CJ)\otimes_TT^\prime=(\SM\boxtimes_TT^\prime)_{[\CK]}(\CJ)$ which is compatible with the restriction for an inclusion of $T$-open subsets. From Proposition \ref{prop-rig} we obtain an isomorphism $\SM_{[\CK]}\boxtimes_TT^\prime\cong(\SM\boxtimes_TT^\prime)_{[\CK]}$. 

(4) After having proven  (1), (2) and (3) it suffices to show that $\SM_{[\CK]}$ satisfies the support condition.  For this we use Lemma \ref{lemma-suppcond}. So let $\CJ^\prime\subset\CJ$ be open subsets. Then the kernel of $\SM_{[\CK]}(\CJ)\to\SM_{[\CK]}(\CJ^\prime)$ is contained in the kernel of $\SM|_{\CK_{\preceq_T}}(\CJ)\to\SM|_{\CK_{\preceq_T}}(\CJ^\prime)$. The latter is $\CZ$-supported on $\pi(\CJ\setminus\CJ^\prime)$, as $\SM|_{\CK_{\preceq_T}}$ satisfies the support condition by Lemma \ref{lemma-admres}. Hence so is the former. 
\end{proof}

\subsection{Exactness and subquotients}

\begin{lemma}\label{lemma-ses2} 
Let $0\to\SA\to\SB\to\SC\to 0$ be a sequence in $\bS$. Then the following statements are equivalent:
\begin{enumerate}
\item The sequence is exact.
\item For any open subset $\CO$ of $\CA_T$ the sequence
$0\to\SA|_{\CO}\to\SB|_{\CO}\to\SC|_{\CO}\to 0
$
is an exact sequence of sheaves on $\CA_T$.  
\item For any locally closed subset $\CK$ of $\CA_T$ the sequence  $
0\to\SA_{[\CK]}\to\SB_{[\CK]}\to\SC_{[\CK]}\to 0
$
is an exact sequence of sheaves on $\CA_T$.  
\item For any $A\in\CA_T$, the sequence
$
0\to\SA_{[A]}\to\SB_{[A]}\to\SC_{[A]}\to 0
$
is an exact sequence of sheaves on $\CA_T$. 
%\item There is an element $s\in\hCS$ such that for any $s$-invariant open subset $\CJ$ of $\CA$ the sequence $0\to\SA(\CJ)\to\SB(\CJ)\to\SC(\CJ)\to 0$ is exact in the category of $\CZ$-modules.
\end{enumerate}
\end{lemma}

\begin{proof} 
By Lemma \ref{lemma-ses1} a sequence of sheaves is exact if and only if all induced local section sequences are exact. This shows that (1) and (2) are equivalent.   (1) and (4) are special cases of (3), and (4) and the snake lemma imply (3), as for each object $\SM$ in $\bS$ its subquotient $\SM_{[\CK]}$  is an extension of the  subquotients $\SM_{[A]}$ with $A\in\CK$ by Lemma \ref{lemma-reshoms}. Suppose (1) holds. Then all sequences of local sections are exact. The snake lemma implies that all induced sequences of local sections of $0\to\SA_{[\CK]}\to\SB_{[\CK]}\to\SC_{[\CK]}\to 0$ are exact, hence another application of Lemma \ref{lemma-ses1} shows that (1) implies (3). 
\end{proof}

\subsection{On morphisms in $\bS$}
Let $\CI$ be a  closed subset of $\CA_T$. Note that $\CI_{\preceq_T}=\Lambda$, where $\Lambda$ is the union of the connected components of $\CA_T$ that have non-empty intersection with $\CI$. For an object $\SM$ in $\bP$ its restriction to $\Lambda$ is a direct summand (cf. Lemma \ref{lemma-supdirsum}). It follows that we can consider $\SM_{[\CI]}$ as a sub-presheaf of $\SM$.  

\begin{lemma}\label{lemma-homs1} 
Let $\SM$ and $\SN$ be objects in $\bS$ and suppose that $\SM$  is supported on $\CK$. Then the image of any morphism $f\colon\SM\to\SN$ is contained in $\SN_{[\CK_{\succeq_T}]}$, i.e. the canonical homomorphism  $\Hom_{\bP}(\SM,\SN_{[\CK_{\succeq_T}]})\to\Hom_\bP(\SM,\SN)$ is an isomorphism. 
\end{lemma}
\begin{proof} Let $\CO$ be the open complement of $\CK_{\succeq_T}$ in $\CA_T$. For an arbitrary open subset $\CJ$ the  diagram 

\centerline{
\xymatrix{
\SM(\CJ)\ar[r]^{f(\CJ)}\ar[d]&\SN(\CJ)\ar[d]\\
\SM(\CO\cap\CJ)\ar[r]^{f(\CO\cap\CJ)}&\SN(\CO\cap\CJ)
}
}
\noindent
commutes. As $\SM$ is supported on $\CK$, $\SM(\CO\cap\CJ)=0$, hence the image of $f(\CJ)$ is contained in the kernel of $\SN(\CJ)\to\SN(\CO\cap\CJ)$, i.e. in the kernel of $\SN|_{\CA_T}(\CJ)\to\SN|_{\CO}(\CJ)$, which is, as $\CO$ is the complement of $\CK_{\succeq_T}$ in $\CA_T$, the object $\SN_{[\CK_{\succeq_T}]}(\CJ)$ by Lemma \ref{lemma-reshoms}. \end{proof}

%Suppose that $\Lambda\in C(\CA_T)$ is a connected component that satisfies $\Lambda\ne\Lambda s$. 
\section{Wall crossing functors}
In this section we review the construction and the basic properties of the wall crossing functor $\vartheta_s\colon\bP_T\to\bP_T$ associated with a simple affine reflection $s$. We recall one of the main results in \cite{FieLanWallCross}, namely that $\vartheta_s$ preserves the subcategory $\bS_T$. We show that $\vartheta_s$ is self-adjoint up to a shift, and that it commutes with the restriction and the subquotient functors associated with $s$-invariant sets. The main result in this section, however, is that the wall crossing functors are exact on $\bS_T$. This, together with their self-adjointness, shows that wall crossing preserves projectivity in $\bS_T$.  
\subsection{The characterizing property of the wall crossing functor} Fix a base ring $T$ and $s\in\hCS$.
 Let $\SM$ be a presheaf of $\CZ=\CZ_T$-modules on $\CA_T$. We define a new presheaf $\epsilon_s\SM$ of $\CZ$-modules on $\CA_T$ by setting 
$$
\epsilon_s\SM(\CJ)=\CZ\otimes_{\CZ^s}\SM(\CJ^\sharp)
$$
for any open subset $\CJ$ of $\CA_T$, with restriction homomorphism $\id\otimes r_{\CJ^\sharp}^{\CJ^{\prime\sharp}}$ for an inclusion $\CJ^\prime\subset\CJ$. In general, $\epsilon_s\SM$ is not an object in $\bP$, even if $\SM$ is, as it does not necessarily satisfy the support condition.

\begin{theorem} \cite[Theorem 7.1 \& Theorem 7.9]{FieLanWallCross} \label{thm-wc} \begin{enumerate}
\item There is an up to isomorphism unique  functor  $\vartheta_s\colon\bP\to\bP$  that admits a natural transformation $\epsilon_s\to\vartheta_s$ such that for any $s$-invariant open subset $\CJ$ of $\CA_T$ and any $\SM\in\bP$ the corresponding homomorphism
$$
\epsilon_s\SM(\CJ)\to\vartheta_s\SM(\CJ)
$$
is an isomorphism of $\CZ$-modules. 
\item The functor $\vartheta_s$ preserves the subcategory $\bS$ of $\bP$.
\end{enumerate}
\end{theorem}
The functor $\vartheta_s$ is called the {\em wall crossing functor} associated with $s$. 
%The functor $\vartheta_s$ is the composition of $\epsilon_s$ and the functor $(\cdot)^+$ of taking the maximal quotient satisfying the support condition. 
\subsection{Properties of $\vartheta_s$}
Recall the homomorphism $\eta_s\colon\CZ\to\CZ$ that we defined in Section \ref{sec-sinvZ}. For a $\CZ$-module $M$ denote by $M^{s-tw}$ the $\CZ$-module that has the same underlying abelian group as $M$, but with the $\CZ$-action twisted by $\eta_s$,  i.e. $z\cdot_{M^{s-tw}} m=\eta_s(z)\cdot_M m$ for $z\in\CZ$ and $m\in M$. 
Let $\Lambda$ be a union of connected components of $\CA_T$. Then $\Lambda s$ is again a union of connected components of $\CA_T$.  Suppose $\Lambda\cap\Lambda s=\emptyset$. Then the map $\gamma_s\colon \Lambda\to\Lambda s$, $A\mapsto As$ is a homeomorphism of topological spaces (cf.  \cite[Lemma 6.1]{FieLanWallCross}). Hence the pull-back functor $\gamma_s^{\ast}$ is an equivalence between the categories of (pre-)sheaves of $\CZ$-modules on $\Lambda s$ and (pre-)sheaves of $\CZ$-modules on $\Lambda$.
Let $\SM$ be an object in $\bP$ supported on $\Lambda s$. We denote by $\gamma_s^{[\ast]}\SM$ the presheaf of $\CZ$-modules on $\CA_T$ that we obtain from $\gamma_s^\ast\SM$ by composing with  the functor $(\cdot)^{s-tw}$. Hence, for any open subset $\CJ$ of $\Lambda$ we have
$$
(\gamma_s^{[\ast]}\SM)(\CJ)=\SM(\CJ s)^{s-tw}.
$$
Then we have the following. 
\begin{lemma}\label{lemma-propwc1} \cite[Lemma 7.3]{FieLanWallCross}
Let $\SM$ be an object in $\bP$ supported on a union of connected components $\Lambda$ that satisfies $\Lambda\cap\Lambda s=\emptyset$. Then 
$$
(\vartheta_s\SM)^{\ol\Lambda}\cong\SM^{\ol\Lambda}\oplus \gamma_s^{[\ast]}(\SM^{{\ol{\Lambda s}}}).
$$
\end{lemma}
We now add results  that we haven't proven in \cite{FieLanWallCross}. 

\begin{lemma} \label{lemma-selfadj} The functor $\vartheta_s\colon \bP\to\bP$ is self-adjoint up to a shift, i.e. 
$$
\Hom_{\bP}(\vartheta_s\SM,\SN)\cong\Hom_{\bP}(\SM,\vartheta_s\SN[2])
$$
functorially in $\SM$ and $\SN$.
\end{lemma}

\begin{proof}
By \cite[Proposition 5.2]{FieTAMS}   the functor $\CZ\otimes_{\CZ^s}\cdot$  is self-adjoint up to a shift, i.e.  $\Hom_{\CZ}(\CZ\otimes_{\CZ^s}M,N)\cong\Hom_{\CZ}(M,\CZ\otimes_{\CZ^s} N[2])$ functorially on the level of $\CZ$-modules. Let $\CJ$ be an $s$-invariant open subset of $\CA_T$.
As $\vartheta_s(\SX)(\CJ)\cong \CZ\otimes_{\CZ^s}\SX(\CJ)$ functorially for any object $\SX$ of $\bP$, there is a functorial   isomorphism 
$$
\phi_\CJ\colon\Hom_{\CZ}((\vartheta_s \SM){(\CJ)}, \SN(\CJ))=\Hom_{\CZ}(\SM(\CJ),(\vartheta_s \SN){(\CJ)[2]})
$$
for all $\SM$ and $\SN$ in $\bP$. Hence, a collection of morphisms $f^{(\CJ)}\colon\vartheta_s\SM(\CJ)\to\SN(\CJ)$  that is compatible with restrictions determines a collection $g^{(\CJ)}\colon\SM(\CJ)\to\vartheta_s\SN(\CJ)[2]$ of morphisms that is compatible with restrictions (in both cases, $\CJ$ runs through all $s$-invariant open subsets).  Proposition  \ref{prop-rig}  shows that these collections uniquely define morphisms $f\colon \vartheta_s\SM\to\SN$ and $g\colon\SM\to\vartheta_s\SN[2]$. 
\end{proof}

\begin{proposition}\label{prop-propwc2} Let $\SM$ be an object in $\bS$. 
\begin{enumerate}
\item Let $\CO\subset \CA_T$ be an $s$-invariant open subset. Then $(\vartheta_s\SM)|_\CO\cong\vartheta_s(\SM|_\CO)$ functorially in $\SM$.
\item  Let $\CK\subset\CA_T$ be a locally closed $s$-invariant subset. Then $
(\vartheta_s\SM)_{[\CK]}\cong \vartheta_s(\SM_{[\CK]})
$
functorially in $\SM$.
%\item $\supp_{\preceq_T}\vartheta_s\SM=(\supp_{\preceq_T}\SM)\cup(\supp_{\preceq_T}\SM)s$. 
%\item The functor $\vartheta_s\colon\bS\to\bS$ is  exact. 
%\item The functor $\vartheta^T_s\colon \bP\to\bP$ is self-adjoint.
\end{enumerate}
\end{proposition}
\begin{proof}
(1) The $s$-invariance of $\CO$ implies that for any $T$-open, $s$-invariant subset $\CJ$ we have \begin{align*}(\vartheta_s\SM)|_\CO(\CJ)&=\vartheta_s(\SM)(\CO\cap\CJ)\cong\epsilon_s\SM(\CO\cap\CJ)\\
%&=\CZ\otimes_{\CZ^s}\SM((\CO\cap\CJ)^\sharp)\\
&=\CZ\otimes_{\CZ^s}\SM(\CO\cap\CJ)\\
&=\CZ\otimes_{\CZ^s}\SM|_\CO(\CJ)=\epsilon_s(\SM|_\CO)(\CJ)\\
&\cong\vartheta_s(\SM|_\CO)(\CJ)
\end{align*} 
functorially in $\SM$. The claim now follows from Proposition \ref{prop-rig} and the fact that $s$-invariant open sets form an admissible family (Lemma \ref{lem-admfam}).

(2) Write $\CK=\CO\setminus\CO^\prime$ with $\CO^\prime\subset\CO$ $s$-invariant and open in $\CA_T$. Applying $\vartheta_s$ to the exact sequence $0\to\SM_{[\CK]}\to\SM|_{\CO}\to\SM|_{\CO^\prime}\to 0$ yields a sequence 
$$
\vartheta_s(\SM_{[\CK]})\to\vartheta_s(\SM|_{\CO})\to\vartheta_s(\SM|_{\CO^\prime}).
$$
By (1)  we can identify the morphism $\vartheta_s(\SM|_{\CO})\to\vartheta_s(\SM|_{\CO^\prime})$ with $(\vartheta_s\SM)|_{\CO}\to(\vartheta_s\SM)|_{\CO^\prime}$. By definition, $(\vartheta_s\SM)_{[\CK]}$ is the kernel of the latter morphism. Hence we obtain a morphism $\delta\colon\vartheta_s(\SM_{[\CK]})\to (\vartheta_s\SM)_{[\CK]}$ such that the diagram

\centerline{
\xymatrix{
&\vartheta_s(\SM_{[\CK]})\ar[d]^\delta\ar[r]&\vartheta_s(\SM|_{\CO})\ar[d]^\sim\ar[r]&\vartheta_s(\SM|_{\CO^\prime})\ar[d]^\sim&\\
0\ar[r]&(\vartheta_s\SM)_{[\CK]}\ar[r]&(\vartheta_s\SM)|_{\CO}\ar[r]&(\vartheta_s\SM)|_{\CO^\prime}\ar[r]&0
}
}
\noindent
commutes. For an $s$-invariant open subset $\CJ$ in $\CA_T$ we have $\vartheta_s (\cdot)(\CJ)\cong\CZ\otimes_{\CZ^s}(\cdot(\CJ))$.  Lemma \ref{lemma-transsinv} implies that  $\CZ\otimes_{\CZ^s}(\cdot)$ is an exact endofunctor on the category of $\CZ$-modules. Hence the sequence
$$
0\to \vartheta_s(\SM_{[\CK]})(\CJ)\to\vartheta_s(\SM|_{\CO})(\CJ)\to\vartheta_s(\SM|_{\CO^\prime})(\CJ)\to 0
$$
is exact (we use Lemma \ref{lemma-ses1}). From this we obtain that $\delta(\CJ)\colon \vartheta_s(\SM_{[\CK]})(\CJ)\to(\vartheta_s\SM)_{[\CK]}(\CJ)$ is an isomorphism. Proposition \ref{prop-rig} now implies that $\delta$ is an isomorphism. 
\end{proof}

Here is another result on how wall crossing interacts with subquotients. 
\begin{lemma} \label{lemma-subquotex} Let $\SM$ be an object in $\bS$ and $A\in\CA$. Denote by $\Lambda$ the connected component of $\CA_T$ that contains $A$. 
\begin{enumerate}
\item Suppose $\Lambda\ne\Lambda s$. Then for any open subset $\CJ\subset\Lambda$ we have an isomorphism $(\vartheta_s\SM)_{[A]}(\CJ)\cong \SM_{[A,As]}(\CJ^\sharp)$ of $\CZ^s$-modules that is functorial in $\SM$ and compatible with restrictions.
\item Suppose that $\Lambda=\Lambda s$. Let $\CJ\subset\Lambda$ be open.
\begin{enumerate}
\item If $A\preceq_T As$, then there is an isomorphism $(\vartheta_s\SM)_{[A]}(\CJ)\cong \SM_{[A,As]}(\CJ^\sharp)$ of $\CZ^s$-modules that is functorial in $\SM$ and compatible with restrictions.
\item If $As\preceq_T A$, then there is an isomorphism  $(\vartheta_s\SM)_{[A]}(\CJ)\cong \SM_{[A,As]}(\CJ^\flat)[-2]$ of $\CZ^s$-modules that is functorial in $\SM$ and compatible with restrictions.
\end{enumerate}

\end{enumerate}
\end{lemma}
\begin{proof} Note that in all cases $\CK=\{A,As\}$ is a locally closed $s$-invariant subset of $\CA_T$. If $\Lambda\ne\Lambda s$, then $\CK$ splits as a topological space into the discrete union of $\{A\}$ and $\{As\}$. Otherwise $A$ and $As$ are $\preceq_T$-comparable, and $\{A,As\}$ is an interval (cf. Lemma 6.1 in \cite{FieLanWallCross}). 

(1) In this case  we have 
$(\vartheta_s\SM)_{[A,As]}\cong(\vartheta_s\SM)_{[A]}\oplus(\vartheta_s\SM)_{[As]}$. As $\CJ\subset\Lambda$ and as $(\vartheta_s\SM)_{[A]}$ is supported on $\Lambda$, we have $(\vartheta_s\SM)_{[A]}(\CJ)=(\vartheta_s\SM)_{[A]}(\CJ^\sharp)$. Hence we can identify $(\vartheta_s\SM)_{[A]}(\CJ)$ with the sub-$\CZ$-module of $(\vartheta_s\SM)_{[A,As]}(\CJ^\sharp)$ that is supported on $\pi(A)\subset\pi(\Lambda)$. Since 
$(\vartheta_s\SM)_{[A,As]}=\vartheta_s(\SM_{[A,As]})$ by Proposition \ref{prop-propwc2}, we have
\begin{align*}
(\vartheta_s\SM)_{[A,As]}(\CJ^\sharp)&=\CZ\otimes_{\CZ^s}\SM_{[A,As]}(\CJ^\sharp)\\
&=\CZ^{\ol{\Lambda\cup\Lambda s}}\otimes_{\CZ^{\ol{\Lambda\cup\Lambda s},s}}\SM_{[A,As]}(\CJ^\sharp).
\end{align*}
We now obtain an isomorphism $(\vartheta_s\SM)_{[A]}(\CJ)\cong\SM_{[A,As]}(\CJ^\sharp)$ as a $\CZ^{\ol{\Lambda\cup\Lambda s},s}$-module, hence as a $\CZ^s$-module, from Lemma \ref{lemma-decZ}. 

(2) Lemma \ref{lemma-sharpflat} states that $B\in\CJ^\sharp\setminus\CJ$ implies $Bs\preceq_T B$. If $A\preceq_TAs$, then $B\in\CJ^{\sharp}\setminus\CJ$ implies $\pi(B)\ne\pi(A)$. By the same lemma $B\in\CJ\setminus\CJ^\flat$ implies  $B\preceq_TBs$. In the case  $As\preceq_T A$  we hence have  $\pi(A)\ne\pi(B)$ for all $B\in\CJ\setminus\CJ^\flat$. As $(\vartheta_s\SM)_{[A]}$ is $\CZ$-supported on $\pi(A)$ (cf. Remark \ref{rem-suppcond2} (2)), Lemma \ref{lemma-suppcond} implies that $(\vartheta_s\SM)_{[A]}(\CJ^\sharp)\to (\vartheta_s\SM)_{[A]}(\CJ)$ is an isomorphism in the case $A\preceq_T As$, and that $(\vartheta_s\SM)_{[A]}(\CJ)\to (\vartheta_s\SM)_{[A]}(\CJ^\flat)$ is an isomorphism if $As\preceq_T A$. Hence we can assume from the beginning that $\CJ=\CJ^\sharp$ in the first case, and $\CJ=\CJ^\flat$ in the second case. In any case, $\CJ$ is then $s$-invariant. 

Denote by $A_1$ and $A_2$ the alcoves such that $A_2\preceq_TA_1$ and $\{A_1,A_2\}=\{A,As\}$. Then there is a short exact sequence
$$
0\to(\vartheta_s\SM)_{[A_1]}\to (\vartheta_s\SM)_{[A,As]}\to (\vartheta_s\SM)_{[A_2]}\to 0
$$
 of sheaves by Lemma \ref{lemma-reshoms} (2). Lemma \ref{lemma-ses1} implies that 
$$
0\to(\vartheta_s\SM)_{[A_1]}(\CJ)\to (\vartheta_s\SM)_{[A,As]}(\CJ)\to (\vartheta_s\SM)_{[A_2]}(\CJ)\to 0
$$
is an exact sequence of $\CZ$-modules, and Remark \ref{rem-suppcond2} implies that $(\vartheta_s\SM)_{[A_1]}(\CJ)$ is the maximal $\CZ$-submodule of $(\vartheta_s\SM)_{[A,As]}(\CJ)$ that is $\CZ$-supported on $x:=\pi(A_1)$, and $(\vartheta_s\SM)_{[A_2]}(\CJ)$ is the maximal quotient of $(\vartheta_s\SM)_{[A,As]}(\CJ)$ that is $\CZ$-supported on $y:=\pi(A_2)$. Note that $y=xs$, so $\CL:=\{x,y\}$ is $s$-invariant.

By Proposition \ref{prop-propwc2}, $(\vartheta_s\SM)_{[A,As]}=\vartheta_s(\SM_{[A,As]})$. As $\CJ$ is $s$-invariant, we deduce
\begin{align*}
(\vartheta_s\SM)_{[A,As]}(\CJ)&=\vartheta_s(\SM_{[A,As]})(\CJ)=\CZ\otimes_{\CZ^s}(\SM_{[A,As]}(\CJ))\\
&=\CZ^{\{x,y\}}\otimes_{\CZ^{\{x,y\},s}}\SM_{[A,As]}(\CJ),
\end{align*}
as $\SM_{[A,As]}(\CJ)$ is $\CZ$-supported on $\{x,y\}$. By Lemma 4.2 in \cite{FieLanWallCross},
$
\CZ^{\{x,y\}}=\{(z_x,z_y)\in T\oplus T\mid z_x\equiv z_y\mod\alpha^\vee\},
$
where $\alpha\in R^+$ is the unique positive root that satisfies $xs=s_{\alpha}x$. Hence $\CZ^{\{x,y\},s}=T(1,1)$. The maximal sub-$\CZ$-module of  $\CZ^{\{x,y\}}$ supported on $x$ is hence  $T(\alpha^\vee,0)$ and the maximal quotient $\CZ$-module of $\CZ^{\{x,y\}}$ that is supported on $y$ is $\CZ^{\{x,y\}}/T(\alpha^\vee,0)$. The claim follows. 
\end{proof}
We can now show that wall crossing functors are exact. \begin{proposition} For all $s\in\hCS$ the functor $\vartheta_s\colon\bS\to\bS$ is exact.
\end{proposition}
\begin{proof} Let $0\to\SA\to\SB\to\SC\to 0$ be a short exact sequence in $\bS$. By Lemma \ref{lemma-ses2} it suffices to show that the sequence $0\to(\vartheta_s\SA)_{[A]}\to(\vartheta_s\SB)_{[A]}\to(\vartheta_s\SC)_{[A]}\to 0$ is exact for all $A\in\CA$. From Lemma \ref{lemma-ses1} we deduce that this is equivalent to showing that the sequence $0\to(\vartheta_s\SA)_{[A]}(\CJ)\to(\vartheta_s\SB)_{[A]}(\CJ)\to(\vartheta_s\SC)_{[A]}(\CJ)\to 0$ is exact for all open subsets $\CJ$. Lemma \ref{lemma-subquotex} allows us to identify this sequence with the sequence
$0\to\SA_{[A,As]}(\tilde \CJ)\to \SB_{[A,As]}(\tilde \CJ)\to\SC_{[A,As]}(\tilde \CJ)\to0$ up to a shift, where $\tCJ$ is either $\CJ^\sharp$ or $\CJ^\flat$. In either case, another application of Lemma \ref{lemma-ses1} and Lemma \ref{lemma-ses2}  shows that the last sequence is exact. 
\end{proof}

\section{Constructing objects in $\bP$
 from $\CZ$-modules}
Let $\CK$ be a section in $\CA_T$. Consider the category $\bZ_\CK$ of root torsion free $\CZ$-modules that are $\CZ$-supported on $\pi(\CK)$, and the full subcategory $\bP_{\CK}$ of $\bP$ that contains all objects that are supported on $\CK$ as presheaves.  In this section we construct an equivalence $\SM_\CK\colon\bZ_{\CK}\xrightarrow{\sim}\bP_{\CK}$.   The  functor $\SM_\CK$ then allows us to introduce for $A\in\CK$  the ``standard object'' $\SV(A)$ in $\bP$ as the image of the corresponding standard object in the category of $\CZ$-modules. We show that the objects $\SV(A)$ are contained in $\bS$ and we prove a result about extensions of standard objects in $\bS$. 
\subsection{The functor $\SM_\CK(\cdot)$} 
Fix  a base ring $T$ and let $\CK\subset\CA_T$ be a section.
\begin{definition}\label{def-Zcat}
\begin{enumerate}
\item
Denote by $\bZ_\CK=\bZ_{T,\CK}$ the full subcategory of  the category of all $\CZ$-modules that are root torsion free and $\CZ$-supported on $\pi(\CK)$.
\item Denote by $\bP_\CK=\bP_{T,\CK}$ the full subcategory of $\bP$ that contains all objects that are as presheaves supported on $\CK$.
\end{enumerate}
\end{definition}
%Recall that {\em supported as a presheaf on $\CK$} means that for any open set $\CJ$ the restriction homomorphism $\SM(\CJ)\to\SM(\CJ\cap\CK)$ is an isomorphism.

Let $M$ be an object in $\bZ_\CK$. 
For an open subset $\CJ$ of $\CA_T$ set $\SM_\CK(M)(\CJ):=M^{\pi(\CK\cap\CJ)}$, and for an inclusion $\CJ^\prime\subset\CJ$ let the restriction be the natural homomorphism $M^{\pi(\CK\cap\CJ)}\to M^{\pi(\CK\cap\CJ^\prime)}$ associated with the inclusion $\pi(\CK\cap\CJ^\prime)\subset\pi(\CK\cap\CJ)$. We obtain a presheaf $\SM_\CK(M)$ of $\CZ$-modules.  
\begin{lemma} \label{lemma-imM} For any object $M$ of $\bZ_\CK$, the presheaf $\SM_\CK(M)$ is an object in $\bP_\CK$.
\end{lemma}
\begin{proof} It is clear that  $\SM_\CK(M)$ is supported, as a presheaf, on $\CK$, and $\SM_\CK(M)(\emptyset)=M^{\emptyset}=0$. Hence $\SM_\CK(M)$ is finitary. It is clearly flabby and it is  root torsion free as $M^{\CL}$ is root torsion free for any $\CL\subset\CV$ (cf. Lemma 4.5 in \cite{FieLanWallCross}).  
It remains to  show that $\SM_\CK(M)$ satisfies the support condition. For this we have to show that the homomorphism $\SM(\CJ)^x\to \SM((\CJ\cap x)_{\preceq_T})^x$ is an isomorphism for all $x\in\CV$ and all $T$-open subsets $\CJ$. By definition, this homomorphism is $(M^{\pi(\CJ\cap\CK)})^x\to (M^{\pi((\CJ\cap x)_{\preceq_T}\cap\CK)})^x$. Now for any subset $\CL$ of $\CV$ we have $(M^{\CL})^x=M^x$, if $x\in\CL$, and $=0$ otherwise. But one easily shows that  $x\in\pi((\CJ\cap x)_{\preceq_T}\cap\CK)$ if and only if $x\in\pi(\CJ\cap \CK)$. \end{proof}

\begin{proposition} \label{prop-equivcat} 
The global section functor $\Gamma$ and the functor $\SM_\CK(\cdot)$ yield mutually inverse equivalences   $\bP_\CK\xrightarrow{\sim}\bZ_\CK$ of categories.
\end{proposition}

 \begin{proof}
 If $\SM$ is in $\bP_\CK$, then $\Gamma(\SM)$ is root torsion free and $\CZ$-supported on $\pi(\CK)$ by Lemma \ref{lemma-supp2}. It is hence an object in $\bZ_\CK$. It remains to show that $\Gamma$ and $\SM_\CK$ are inverse functors. 
For $M\in\bZ_\CK$ we have $\Gamma(\SM_\CK(M))=M^{\pi(\CA\cap\CK)}=M^{\pi(\CK)}=M$ functorially.  Hence $\Gamma\circ\SM_\CK\cong\id_{\bZ_\CK}$. Let $\SM$ be an object in $\bP_\CK$ and let $\CJ$ be  open in $\CA_T$. By Lemma \ref{lemma-homs2} the restriction homomorphism induces an isomorphism $\Gamma(\SM)^{\pi(\CJ\cap\CK)}\xrightarrow{\sim}\SM(\CJ)$, which yields an isomorphism $\SM_\CK(\Gamma(\SM))\xrightarrow{\sim}\SM$ that is clearly functorial. Hence $\SM_\CK\circ\Gamma\cong\id_{\bP_\CK}$.
\end{proof}

\subsection{$\SM_\CK$, base change and wall crossing}
The following result introduces more classical wall crossing and base change functors on the category $\bZ_\CK$. 
\begin{lemma}
\begin{enumerate}
\item For any $s\in\hCS$ such that $\CK$ is $s$-invariant, the functor $\CZ\otimes_{\CZ^s}(\cdot)$ is a functor from $\bZ_\CK$ to $\bZ_\CK$.
\item For any flat homomorphism $T\to T^\prime$ of base rings, the functor $(\cdot)\otimes_TT^\prime$ is a functor from $\bZ_{T,\CK}$ to $\bZ_{T^\prime,\CK}$. 
\end{enumerate}\end{lemma}

\begin{proof}
It follows from Lemma \ref{lemma-transsinv} that $\CZ\otimes_{\CZ^s} M$ is root torsion free and $\CZ$-supported on $\CK$ for any object $M$ in $\bZ_\CK$, hence  (1). 
As the functor $(\cdot)\otimes_TT^\prime$ for a flat homomorphism $T\to T^\prime$ preserves root torsion freeness and the $\CZ$-support, (2) follows. 
\end{proof} 

We now show  that the functor $\SM_\CK$ intertwines  the  wall crossing and the base change functors on $\bZ_\CK$  and on $\bP_\CK$.  

\begin{lemma} \label{lemma-firstpropM} 
Let $M$ be an object in $\bZ_\CK$. 
\begin{enumerate} 
%\item $\SM_\CK(M)$ is an object in $\bZ_\CK$. 
\item  For a flat homomorphism $T\to T^\prime$ we have $\SM_\CK(M)\boxtimes_TT^\prime\cong\SM_\CK(M\otimes_TT^\prime)$ functorially in $M$.
 \item Suppose  that $\CK$ is $s$-invariant.  Then $\vartheta_s(\SM_\CK(M))\cong\SM_\CK(\CZ\otimes_{\CZ^s}M)$.
\end{enumerate}
\end{lemma}

\begin{proof}
(1)  For any $T$-open subset $\CJ$ we have isomorphisms 
\begin{align*}
\SM_{\CK}(M)\boxtimes_TT^\prime(\CJ)&\cong \SM_\CK(M)(\CJ)\otimes_TT^\prime\\&=M^{\pi(\CJ\cap\CK)}\otimes_TT^\prime\\&=(M\otimes_TT^\prime)^{\pi(\CJ\cap\CK)}\\&=\SM_{\CK}(M\otimes_TT^\prime)(\CJ)
\end{align*}
and each isomorphism is compatible with restriction homomorphisms. The fact that the $T$-open sets form a $T^\prime$-admissible family (Lemma \ref{lem-admfam}) and Proposition \ref{prop-rig} prove the claim.

(2) Let $\CJ\subset\CA$ be $s$-invariant and $T$-open. Then
\begin{align*}
\vartheta_s\SM_\CK(M)(\CJ)
&=\CZ\otimes_{\CZ^s} \SM_\CK(M)(\CJ)\\
&=\CZ\otimes_{\CZ^s} M^{\pi(\CK\cap\CJ)}\\
&=(\CZ\otimes_{\CZ^s} M)^{\pi(\CK\cap\CJ)}=\SM_\CK(\CZ\otimes_{\CZ^s} M)(\CJ).
\end{align*}
(For the third step we used Lemma \ref{lemma-transsinv}.) Now the claim follows from the fact that the $s$-invariant subsets form a $T$-admissible family (Lemma \ref{lem-admfam}) and Proposition \ref{prop-rig}. 
\end{proof}

\subsection{The standard objects} 

Our first example of sheaves constructed using the functor $\SM_\CK$ are the following {\em standard} (or skyscraper) sheaves. 
Let $A\in\CA$ and let $\CK\subset\CA_T$ be a section with $A\in\CK$ (note that the special sections $\Lambda_\mu$ cover $\CA$, so there always is a section containing $A$). %Denote by $\Delta(\pi(A))$ the $\CZ$-module that is free of rank $1$ as a $T$-module and on which $(z_x)\in\CZ$ acts by scalar multiplication with $z_{\pi(A)}$. This is clearly an object in $\bZ_\CK$. 
Define
$$
\SV(A)=\SV_T(A):=\SM_\CK(\CZ^{\pi(A)}).
$$
Explicitely,  this is given as follows. Let $\CJ\subset\CA_T$ be open. Then
$$
\SV(A)(\CJ)=\begin{cases}
\CZ^{\pi(A)},&\text{ if $A\in\CJ$},\\
0,&\text{ if $A\not\in\CJ$}
\end{cases}
$$
together with the obvious non-trivial restriction homomorphisms. In particular, the object $\SV(A)$ does not depend on the choice of $\CK$. 
\begin{lemma} Let $A\in\CA$.\begin{enumerate}
\item $\SV(A)$ is an object in $\bS$.
\item For a flat homomorphism $T\to T^\prime$ of base rings we have $\SV_T(A)\boxtimes_TT^\prime=\SV_{T^\prime}(A)$.
\end{enumerate}
\end{lemma}
\begin{proof}
From the above it is apparent that  $\SV(A)$ is a root reflexive sheaf (note that $\CZ^{\pi(A)}\cong T$ as a $T$-module).  Moreover, using  Lemma \ref{lemma-firstpropM} we obtain $\SV_T(A)\boxtimes_TT^\prime=\SM_{\CK}(\CZ_T^{\pi(A)})\boxtimes_TT^\prime\cong\SM_{\CK}(\CZ_T^{\pi(A)}\otimes_TT^\prime)=\SM_{\CK}(\CZ_{T^\prime}^{\pi(A)})=\SV_{T^\prime}(A)$, hence (2). It follows that $\SV(A)$ is actually an object in $\bS$. 
\end{proof}

\subsection{Extensions of standard objects}
We finish this section with a result on the extension structure of the $\SV(A)$'s. 
Let $A,B\in\CA$.
\begin{lemma} \label{lemma-ses3} Let $0\to\SV(A)\xrightarrow{f}\SX\xrightarrow{g}\SV(B)\to 0$ be a non-split exact sequence in $\bS$. Then $B\prec_T A$ and there  exists some $\alpha\in R^+_T$ with $\pi(B)=s_\alpha\pi(A)$. 
\end{lemma}
\begin{proof} We refer to the sequence in the above claim by $(\ast)$. By Lemma \ref{lemma-ses2}, the induced sequence $0\to\SV(A)|_{\preceq_TA}\xrightarrow{f|_{\preceq_TA}}\SX|_{\preceq_TA}\xrightarrow{g|_{\preceq_TA}}\SV(B)|_{\preceq_TA}\to 0$ is exact as well. Note that $\SV(A)\cong\SV(A)|_{\preceq_TA}$. If $B\not\preceq_TA$, then $\SV(B)|_{\preceq_TA}=0$ and hence $f|_{\preceq_TA}$ is an isomorphism, hence so is the composition $\SV(A)\xrightarrow{f}\SX\xrightarrow{\pi}\SX|_{\preceq_TA}$, where $\pi$ is the natural epimorphism. This contradicts the fact that the sequence $(\ast)$ is non-split. Hence $B\preceq_TA$. 

Now we show $\pi(A)\ne\pi(B)$ (which of course then also yields $B\prec_TA$). So suppose the contrary. If $\pi(A)=\pi(B)=x$, then all sheaves appearing in the sequence $(\ast)$ are sheaves of $\CZ^x=T$-modules. We define now a subsheaf $\SY$ of $\SX$ such that $g|_\SY\colon\SY\to\SV(B)$ is an isomorphism. Let $\CJ$ be open in $\CA_T$. If $A\not\in\CJ$, set  $\SY(\CJ):=\SX(\CJ)$, if $B\not\in\CJ$, set $\SY(\CJ)=0$ (note that if $A,B\not\in\CJ$, then $\SX(\CJ)=0$). Now consider $\CJ_0=\{\preceq_TA\}\cup\{\preceq_TB\}$. Then $\SV(B)(\CJ_0)=T$ is a projective $\CZ^x$-module, hence the sequence
$$0\to\SV(A)(\CJ_0)\xrightarrow{f(\CJ_0)}\SX(\CJ_0)\xrightarrow{g(\CJ_0)}\SV(B)(\CJ_0)\to 0
$$
splits (non-canonically). Fix such a splitting, i.e. a homomorphism $h\colon \SV(B)(\CJ_0)\to\SX(\CJ_0)$ with the property $g(\CJ_0)\circ h=\id_{\SV(B)(\CJ_0)}$, and let $\SY(\CJ_0)$ be the image of $h$. If $\CJ$ is an arbitrary open subset that contains $A$ and $B$, then $\CJ_0\subset\CJ$ and the restrictions from $\CJ$ to $\CJ_0$ of $\SV(A)$, $\SX$ and $\SV(B)$ are isomorphisms. Hence we can define $\SY(\CJ)$ as the preimage of $\SY(\CJ_0)$ in $\SX(\CJ)$. It is clear from the construction that $\SY$ is indeed a sub(pre-)sheaf and that $g|_\SY$ is an isomorphism of (pre-)sheaves. But this means that the sequence in the claim splits, contrary to the assumption. Hence $\pi(A)=\pi(B)$ cannot be true.

Now let $\CJ$ be $T$-open. Then $\CJ$ is $T^{\alpha}$-open for any $\alpha\in R^+$, as the $T^{\alpha}$-topology is finer than the $T$-topology. As $\SX(\CJ)$ is root reflexive, it is the intersection of $\SX(\CJ)\otimes_TT^{\alpha}$ over all $\alpha\in R^+$. By Proposition \ref{prop-defbox}, the latter object identifies functorially with $\SX\boxtimes_TT^{\alpha}(\CJ)$ for all $\CJ$ open in $\CA_T$. If $A$ and $B$ were contained in distinct connected components of $\CA_{T^\alpha}$ for all $\alpha\in R^+$, then $\SX\boxtimes_TT^{\alpha}$ would split canonically into a direct sum of subsheaves supported on distinct  connected components, and hence we would obtain a canonical splitting of $\SX(\CJ)$ that is compatible with restrictions. Hence $\SX$ would split, contrary to our assumption. Hence there is some $\alpha$ such that $A$ and $B$ are contained in the same connected component of $\CA_{T^\alpha}$. As $\pi(A)\ne\pi(B)$,  this implies $\pi(B)=s_\alpha\pi(A)$ and $\alpha\in R_T^+$. 
\end{proof}

\section{Verma flags and graded characters}
We now introduce the subcategory $\bB_T$ of $\bS_T$ that contains all objects whose local sections are not only reflexive, but actually (graded) free modules over the base ring $T$. It is the category of ``objects admitting a Verma flag''. It can also be defined as the subcategory in $\bS_T$ of objects admitting a finite filtration with subquotients being isomorphic to various standard sheaves $\SV(A)$, possibly shifted in degree. To any object in $\bB_T$  we can associate a (possibly graded) character. We show that the wall crossing functors preserve the category $\bB_T$, and we study their effect on graded characters. 

\subsection{Verma flags} 
Fix a base ring $T$ and let $\SM$ be an object in $\bS=\bS_T$. 
\begin{definition}\label{def-Vermaflag} \begin{enumerate}
\item
We say that $\SM$ {\em admits a Verma flag} if for any flat homomorphism $T\to T^\prime$ and for any $T^\prime$-open subset $\CJ$ the object $\SM\boxtimes_TT^\prime(\CJ)$ is a (graded) free  $T^\prime$-module of finite rank. 
\item We denote by $\bB=\bB_T$ the full subcategory of $\bS$ that contains all objects that admit a Verma flag. 
\item For an object $\SM$ in $\bB$ we define its {\em $\preceq_T$-support} by 
$$
\supp_{\preceq_T}\SM:=\{A\in\CA\mid \SM_{[A]}\ne 0\}.
$$
\end{enumerate}
\end{definition}
In the definition of a base ring $T$ we assume that every projective $T$-module is free. Hence it follows that the category $\bB$ is stable under taking direct summands. 

\begin{lemma} \label{lemma-Vflag} 
Let $\SM$ be an object in $\bS$.  The following are equivalent.
\begin{enumerate}
\item $\SM$ admits a Verma flag. 
\item $\SM|_\CO$ admits a Verma flag for  any open subset $\CO$ of $\CA_T$.
\item $\SM_{[\CK]}$ admits a Verma flag for any locally closed subset $\CK$ of $\CA_T$.
\item $\SM_{[A]}$  admits a Verma flag for  any $A\in\CA$.
\end{enumerate}
\end{lemma}

\begin{proof}
(1) is a special case of (2), and as the local sections of $\SM|_{\CO}$ are local sections of $\SM$,  (1) implies (2). (1) is also a special case of (3). Suppose (1) holds. Then the short exact sequence
$$
0\to\SM_{[\CK]}(\CJ)\to\SM|_{\CK_{\preceq_T}}(\CJ)\to\SM|_{\CK_{\preceq_T}\setminus\CK}(\CJ)\to 0
$$ 
splits as a sequence of $T$-modules, as the $T$-module on the right is (graded) free  (as (1) implies (2)). As the $T$-module in the middle is (graded) free as well, the $T$-module on the left is also (graded) free (note that one of the assumptions on a base ring is that any projective $T$-module is free). Hence (1) implies (3). (4) is a special case of (3) and it implies (3) as $\SM_{[\CK]}$ is a successive extension of the $\SM_{[A]}$ with $A\in\CK$. 
\end{proof}

\begin{lemma}\label{lemma-Vflag2} Let $\SM$ be an object in $\bS$ and $A\in\CA$ and suppose that $\SM=\SM_{[A]}$. Then the following are equivalent. 
\begin{enumerate}
\item  $\SM$ admits a Verma  flag.
\item $\SM$ is isomorphic to a finite direct sum of (shifted) copies of $\SV(A)$.
\item $\Gamma(\SM)$ is (graded) free as a $T$-module of finite rank.
\end{enumerate}
If either of the above statements holds, then the (graded) multiplicity of $\SV(A)$ in $\SM$ coincides with the (graded) multiplicity of $T$ in $\Gamma(\SM)$. 
%If either of the above conditions is satisfies, then $\SM\cong p(\Gamma(\SM))\cdot\SV(A)$.
\end{lemma}
\begin{proof} First note that $\SM=\SM_{[A]}$  implies that $\SM(\CJ)=\Gamma(\SM)$ if $A\in\CJ$, and $\SM(\CJ)=0$ otherwise. Hence (1) and (3) are equivalent. Note that $\SM_{[A]}$ is, by Remark \ref{rem-suppcond2}, a sheaf of $\CZ^{\pi(A)}$-modules and supported as a presheaf on $\{A\}$. Using Proposition \ref{prop-equivcat} we deduce that  $\Gamma(\SM_{[A]})$ is graded free as a $T$-module if and only if there is an isomorphism between $\SM_{[A]}$ and a direct sum of (shifted in degree) copies of $\SV(A)$. Hence (2) and (3) are equivalent, and the last statement follows as well. \end{proof}
Let $s$ be a simple affine reflection.
\begin{lemma} The wall crossing functor $\vartheta_s$ preserves the subcategory $\bB$ of $\bS$.
\end{lemma} 
\begin{proof} This follows from the fact that, if $\SM$ admits a Verma flag, then  $\SM_{[A,As]}(\CJ)$ is a (graded) free $T$-module for any $A\in\CA$ and all open subset $\CJ$ by Lemma \ref{lemma-Vflag}. Hence, by Lemma \ref{lemma-subquotex}, so is $(\vartheta_s\SM)_{[A]}(\CJ)$ for any open subset $\CJ$ and any $A\in\CA$, so $\vartheta_s\SM$ admits a Verma flag, again by Lemma \ref{lemma-Vflag}. 
\end{proof}

We add a rather technical statement that we need later on.

\begin{lemma}\label{lemma-sheafrefl} Let $\SM$ be an object in $\bP$. Then $\SM$ is an object in $\bS$ and admits a Verma flag if and only if the following holds.
\begin{enumerate}
\item For any flat base change $T\to T^\prime$, the local sections of the presheaf $\SM\boxtimes_TT^\prime$ on $T^\prime$-open subsets form a (graded) free $T^\prime$-module of finite rank. 
\item For any flat base change $T\to T^\prime$ such that $T^\prime$ is either generic or subgeneric, $\SM\boxtimes_TT^\prime$ is a sheaf.
\end{enumerate}
\end{lemma}
\begin{proof} The two properties above are part of the definition of $\bS$ and of objects admitting a Verma flag, hence we only have to show the ``if''-part. So assume that the two properties hold. Let $T\to T^\prime$ be a flat homomorphism of base rings. Let $\CJ$ be a $T^\prime$-open subset. Assumption (1) implies that $\SM\boxtimes_TT^\prime(\CJ)=\bigcap_{\alpha\in R^+}\SM\boxtimes_TT^\prime(\CJ)\otimes_{T^\prime}T^{\prime\alpha}$. Now fix $\alpha\in R^+$. If $T^{\prime\alpha}$ is generic, then we have a homomorphism $T^{\emptyset}\to T^{\prime\alpha}$ and 
$$
\SM\boxtimes_TT^{\prime}(\CJ)\otimes_{T^\prime}T^{\prime\alpha}=\SM\boxtimes_TT^{\emptyset}(\CJ)\otimes_{T^\emptyset}T^{\prime\alpha}.
$$
(Note that  $\CJ$ is $T^{\emptyset}$-open as well.) If $T^{\prime\alpha}$ is subgeneric, then we have a homomorphism $T^{\alpha}\to T^{\prime\alpha}$ and 
$$
\SM\boxtimes_TT^{\prime}(\CJ)\otimes_{T^\prime}T^{\prime\alpha}=\SM\boxtimes_TT^{\alpha}(\CJ)\otimes_{T^\alpha}T^{\prime\alpha}.
$$
Using (2) we deduce that $\SM\boxtimes_TT^{\prime}(\CJ)$ is the intersection of local sections over $\CJ$ of {\em sheaves}. It follows that $\SM\boxtimes_TT^\prime$ is a sheaf as well. 
\end{proof}

\subsection{Graded multiplicities}
We will only consider graded characters in the case $T=S$. 
Recall the definition of graded characters of $S$-modules in Section \ref{subsec-quotZ}. %In general, if $M$ is a graded object in a $\CZ$-graded category and $f\in\DZ_{\ge0}[v]$, given by $f(v)=\sum_{n\in\DZ}a_n v^{-n}$ we denote by $f\cdot M$ the object $\bigoplus_{n\in\DZ} M[n]^{a_n}$. 

%In case $T$ is non-graded, then $M\cong\bigoplus rT$ for some $r\in\DZ_{\ge 0}$, and we set $p(M)=r\in\DZ$ and think of $v$ as being not a variable, but $v=1$. The following formulas then also make sense in this case when we ignore the shift functors $(\cdot)[l]$.

Suppose $\SM$ is an object in $\bB=\bB_S$.  By Lemma \ref{lemma-Vflag} and Lemma \ref{lemma-Vflag2},  for any $A$ there is an isomorphism $\SM_{[A]}\cong\bigoplus_{n\in\DZ}\SV(A)[n]^{r_n}$ for some well defined  $r_n\in\DZ_{\ge0}$ (almost all of them $=0$). Define
$$
(\SM:\SV(A)):=\sum_{n\in\DZ} r_n v^{-n}\in\DZ[v, v^{-1}].
$$
In particular, $(\SM:\SV(A))=(\SM_{[A]}:\SV(A))$ and $(\SV(A)[n]:\SV(A))=v^{-n}$ for all $n\in\DZ$. Moreover, from  Lemma \ref{lemma-Vflag2} we deduce
$$
(\SM:\SV(A))=p(\Gamma(\SM_{[A]}))\in\DZ[v,v^{-1}].
$$

Now fix $s\in\hCS$.
\begin{lemma}\label{lemma-charcross}  For all $\SM$ in $\bB$ we have
$$
(\vartheta_s\SM:\SV(A))=\begin{cases}
v^2(\SM:\SV(A))+v^2(\SM:\SV(As)),&\text{ if $As\preceq_S A$}\\
(\SM:\SV(A))+(\SM:\SV(As)),&\text{ if $A\preceq_S As$.}
\end{cases}
$$
\end{lemma}
\begin{proof} We have already shown that $\vartheta_s\SM$ admits a Verma flag.
It follows from Lemma \ref{lemma-subquotex} that $\Gamma((\vartheta_s\SM)_{[A]})$ is isomorphic, as a graded $S$-module, to $\Gamma(\SM_{[A,As]})[r]$, where $r=0$ if $A\preceq_S As$ and $r=-2$ if $As\preceq_S A$. The multiplicity statement hence follows from Lemma \ref{lemma-Vflag2} and the fact that $\SM_{[A,As]}$ is an extension of $\SM_{[A]}$ and $\SM_{[As]}$.  
\end{proof}

\section{Projective objects in $\bS$}
In this section we prove that for any $A\in\CA$ there is an up to isomorphism unique  indecomposable projective object $\SB(A)$ in $\bS$ that admits an epimorphism onto $\SV(A)$. We also show that for any $A\in\CA$, the object $\SB(A)$ admits a Verma flag. The method of proof is very similar to the construction of projectives in the category $\CO$ over a semisimple complex Lie algebra, or in the category of $G_1T$-modules in case $G$ is a modular algebraic group (cf. \cite{AJS}). We first construct quite explicitely a set of {\em special projectives}. We then use the fact that the wall crossing functors are exact and self-adjoint to produce further projectives. We prove some results on projectives  in subgeneric situations, and we finish the section with a rather technical lemma that concerns extensions of standard objects in projective objects.

\subsection{Special sheaves} Let $\mu\in X$ and let $\Lambda\in C(\CA_T)$ be a connected component. Recall the section $\Lambda_\mu$ that was introduced in Section \ref{subsec-toponV}.
We now apply the functor $\SM_{\Lambda_\mu}$ to the $\CZ$-module $\CZ^{\ol\Lambda}$ (this is clearly a root torsion free $\CZ$-module that is $\CZ$-supported on $\ol{\Lambda_\mu}=\ol\Lambda$).

\begin{proposition} \label{prop-specsheaf} The object  $\SM_{\Lambda_\mu}(\CZ^{\ol\Lambda})$ is an object in  $\bB$ and  for any $A\in\CA$ we have
$$
(\SM_{\Lambda_\mu}(\CZ^{\ol\Lambda}):\SV(A))=\begin{cases} 0,&\text{ if $A\not\in\Lambda_\mu$}\\
v^{2n_A},&\text{ if $A\in\Lambda_\mu$},\\
\end{cases}
$$
where $n_A$ is the number of $\alpha\in R^+$ with $\alpha\downarrow A\in\Lambda_\mu$. 
\end{proposition}

\begin{proof} 
We already know that $\SM_{\Lambda_\mu}(\CZ^{\ol\Lambda})$ is an object in $\bP$.  Let $\CJ$ be an open subset of $\CA_T$. Then   $\SM_{\Lambda_\mu}(\CZ^{\ol\Lambda})(\CJ)=\CZ^{\pi(\CJ\cap\Lambda_\mu)}$. As $\pi(\CJ\cap\Lambda_\mu)$ is a $(T,\mu)$-open subset of $\CV$ for any $T$-open subset $\CJ$ (and is contained in the connected component $\ol\Lambda_\mu$ of  $\CV_{T,\mu}$), it follows from the description in Proposition  \ref{prop-Zlocfree} that $\SM_{\Lambda_\mu}(\CZ^{\ol\Lambda})$ satisfies the sheaf condition and that every local section is a free $T$-module of finite rank.  For any flat base change $T\to T^\prime$ we have, by Lemma \ref{lemma-firstpropM}, $\SM_{\Lambda_\mu}(\CZ^{\ol\Lambda}_T)\boxtimes_TT^\prime=\SM_{\Lambda_\mu}(\CZ^{\ol\Lambda}\otimes_TT^\prime)=\SM_{\Lambda_\mu}(\CZ^{\ol\Lambda}_{T^\prime})$. Note that $\CZ^{\ol\Lambda}_{T^\prime}$ now splits according to the decomposition of $\ol\Lambda$ into connected components of $\CV_{T^\prime,\mu}$, so again we can apply Proposition \ref{prop-Zlocfree} and deduce that  $\SM_{\Lambda_\mu}(\CZ^{\ol\Lambda})\boxtimes_TT^\prime$ is a sheaf (on $\CA_{T^\prime}$) with local sections being graded free of finite rank. In particular, $\SM_{\Lambda_\mu}(\CZ^{\ol\Lambda})$ is an object in $\bB$. 
 If $A\not\in\Lambda_\mu$, then $\SM_{\Lambda_\mu}(\CZ^{\ol\Lambda})_{[A]}=0$ by construction. For  $A\in\Lambda_\mu$, the statement about the Verma multiplicity follows from the statement about the graded characters in Proposition \ref{prop-Zlocfree} and Lemma \ref{lemma-Vflag2}. 
\end{proof}

\subsection{Special projectives}  
Let $\CK$ be a section and set $\CI:=\CK_{\succeq_T}$. Then $\CI$ is a closed subset of $\CA_T$.

\begin{lemma}\label{lemma-specproj} Let $M$ be an object in $\bZ_\CK$. Suppose that $\SM_{\CK}(M)$ is an object in $\bS$  and that $M$ is  projective in the category of $\CZ$-modules. Then $\SM_{\CK}(M)$ is a projective object in $\bS$. \end{lemma}
\begin{proof} Let $0\to\SA\to\SB\to\SC\to 0$ be a short exact sequence in $\bS$.  By Lemma \ref{lemma-ses1} and  Lemma \ref{lemma-ses2} the short exact sequence stays exact after applying the functor $\Gamma(\cdot_{[\CI]})$. Lemma \ref{lemma-homs1} yields an isomorphism  $\Hom_\bS(\SM_{\CK}(M),(\cdot))\cong\Hom_\bP(\SM_{\CK}(M),(\cdot)_{[\CI]})$, and Lemma \ref{lemma-homs2} yields another isomorphism $\Hom_\bP(\SM_{\CK}(M),(\cdot)_{[\CI]})\cong\Hom_\CZ(M,\Gamma(\cdot_{[\CI]}))$. Note that all isomorphisms above are functorial. The projectivity of $M$ in the category of $\CZ$-modules now yields the  claim. 
\end{proof}
\begin{remark} \label{rem-specproj} For $\mu\in X$ consider  $\SM_{\Lambda_\mu}(\CZ^{\ol\Lambda})$. This  is an object in $\bS$ by Proposition  \ref{prop-specsheaf} and $\CZ^{\ol\Lambda}$ is a direct summand of $\CZ$, hence  projective in the category of $\CZ$-modules. So by the above $\SM_{\Lambda_\mu}(\CZ^{\ol\Lambda})$ is a projective object in $\bS$.\end{remark}

\subsection{Constructing  projectives via wall crossing functors}

Now suppose that $T=S$. Recall that we consider $S$ as a graded algebra, so all objects in $\bP$, $\bS$, and $\bB$ are graded. 
\begin{theorem}\label{thm-projs} Let $A\in\CA$. There is an up to isomorphism unique object $\SB(A)$ in $\bS$ with the following properties:
\begin{enumerate}
\item $\SB(A)$ is indecomposable and projective in $\bS$.
\item $\SB(A)$ admits an epimorphism $\SB(A)\to\SV(A)$.
\end{enumerate}
Moreover, the object $\SB(A)$ admits a  Verma flag for all $A\in\CA$. 
\end{theorem}

\begin{proof} 
We first show that there is at most one object (up to isomorphism) with the properties listed in (1) and (2).
Suppose that we have found an object $\SB(A)$ that has the stated properties, and suppose that $\SP$ is a projective object in $\bS$ that has an  epimorphism onto $\SV(A)$. Then the projectivity of $\SB(A)$ and $\SP$ implies that we can find homomorphisms $f\colon \SB(A)\to\SP$ and $g\colon \SP\to\SB(A)$ such that the diagram 

\centerline{
\xymatrix{
\SB(A)\ar[dr]\ar[r]^f&\SP\ar[d]\ar[r]^g&\SB(A)\ar[dl]\\
&\SV(A)&
}
}
\noindent 
commutes. The composition $g\circ f$ must be an automorphism, as it cannot be nilpotent and $\SB(A)$ is indecomposable (we use the Fitting decomposition in the category of graded sheaves of $S$-modules on $\CA_S$). Hence $\SB(A)$ is a direct summand of $\SP$. If the latter is indecomposable, then $\SB(A)\cong\SP$, and hence we have proven the uniqueness statement. 
So we are left with showing that an object $\SB(A)$ in $\bS$ with the desired properties exists. For this, we do not need the assumption that $T=S$, so we prove the existence for arbitrary base rings $T$.

First suppose that $A$ is a $\preceq_T$-minimal element in a section $\Lambda_\mu$ (for some $\Lambda\in C(\CA_T)$ and $\mu\in X$). In this case, consider the object $\SB(A):=\SM_{\Lambda_\mu}(\CZ^{\ol\Lambda})$. It is an object in $\bB$ by Proposition \ref{prop-specsheaf}, it is  projective  by Remark \ref{rem-specproj}. As $A$ is the $\preceq_T$-smallest element in $\Lambda_\mu$, the canonical homomorphism $\CZ^{\ol\Lambda}\to\CZ^{\pi(A)}$ yields a homomorphism $\SB(A)\to\SV(A)$ with kernel $\SB(A)_{[\succ_T A]}$. Hence this is an epimorphism in $\bB$. As $\CZ^{\ol\Lambda}$ is indecomposable as  a $\CZ$-module, $\SB(A)$ is indecomposable in $\bB$.

Now let $A\in\CA$ be an arbitrary alcove and let $A^\prime\in\CA$ be an alcove which is $\preceq_T$-minimal in $\Lambda_\mu$ for some $\mu\in X$ and some $\Lambda\in C(\CA_T)$ and has the property that $\lgl\lambda_{A^\prime}-\lambda_A,\alpha^\vee\rgl>0$ for all $\alpha\in R^+$. By Lemma \ref{lemma-specalc} we can find simple affine reflections $s_1$,\dots,$s_n$ in $\hCS$ such that  $A=A^\prime s_1\cdots s_n$ and $ A^\prime s_1\cdots s_{i-1}$ and $ A^\prime s_1\cdots s_{i}$ are either $\preceq_T$-incomparable, or $ A^\prime s_1\cdots s_{i}\prec_T A^\prime s_1\cdots s_{i-1}$.  As the wall crossing functors preserve projectivity and the subcategory $\bB$,  the object $\SQ=\vartheta_{s_n} \cdots\vartheta_{s_1}\SB(A^\prime)$ is projective in $\bS$ and it is contained in $\bB$. 
%As the translation functors preserve the category $\bB$ and as $\SB(A_{w_0}+\lambda)$ is contained in $\bB$, the object $\SQ$ is in $\bB$ as well. 
By Lemma \ref{lemma-charcross}  and the construction of the sequence $s_1$,\dots, $s_n$, $A$ is a $\preceq_T$-minimal element in $\supp_{\preceq_T} \SQ$, and $(\SQ:\SV(A))=1$. Hence there is an epimorphism $\SQ\to\SV(A)$. Any direct summand of $\SQ$ is an object in $\bB$ as well (cf. the remark following Definition \ref{def-Vermaflag}). Hence there is an indecomposable direct summand in $\SQ$ that maps surjectively onto $\SV(A)$. This proves the existence part. 
\end{proof}

\begin{remark}\label{rem-strucproj} It follows from the construction above that $(\SB(A),\SV(A))=1$ and $(\SB(A): \SV(B))\ne 0$ implies $A\preceq_T B$. 
\end{remark}

\subsection{Projectives in the subgeneric case}

%In this section
Assume that $T$ is subgeneric with $R_T^+=\{\alpha\}$. All results of this subsection are under this hypothesis. Let $A\in\CA$ and let $\Lambda$ be the connected component of $\CA_T$ containing $A$. Then $\alpha\uparrow A\in\Lambda$ and $\ol{\Lambda}=\pi(\Lambda)=\pi(\{A,\alpha\uparrow A\})$. Moreover,  Lemma \ref{lemma-upmin} implies that $\{A,\alpha\uparrow A\}$ is a section in $\CA_T$. Define  
$$
\SQ(A)=\SQ_T(A):=\SM_{\{A,\alpha\uparrow A\}}(\CZ^{\ol\Lambda}).
$$  
More explicitely, 
$$
\SQ(A)(\CJ)=
\begin{cases} 
\CZ^{\ol\Lambda},&\text{ if $A,\alpha\uparrow A\in\CJ$},\\
\CZ^{\pi(A)},&\text{ if $\alpha\uparrow A\not\in\CJ$, $A\in\CJ$},\\
0, &\text{ if $A\not\in\CJ$}.
\end{cases}
$$
\begin{lemma}  
\begin{enumerate}
\item $\SQ(A)$ is a projective object in $\bS$ and it is contained in $\bB$. 
\item There is a short exact sequence $0\to\SV(\alpha\uparrow A)[-2]\to\SQ(A)\to\SV(A)\to 0$.
\end{enumerate}
\end{lemma}
\begin{proof} 
From the above explicit description it is clear that $\SQ(A)$ is a sheaf on $\CA_T$ with local sections that are graded free as $T$-modules, and that there is a short exact sequence in the category of sheaves on $\CA_T$ as claimed in (2). Let $T\to T^\prime$ be a flat homomorphism of base rings. If $T^\prime$ is again subgeneric, then  $\SQ_T(A)\boxtimes_TT^\prime=\SQ_{T^\prime}(A)$ (using Lemma \ref{lemma-firstpropM}). If $T^\prime$ is generic, then $\CZ_{T^\prime}^{\ol\Lambda}=\CZ_{T^\prime}^{\pi(A)}\oplus\CZ_{T^\prime}^{\pi(\alpha\uparrow A)}$ and hence $\SQ_T(A)\boxtimes_TT^\prime\cong\SV_{T^\prime}(A)\oplus \SV_{T^\prime}(\alpha\uparrow A)$. In any case, $\SQ_T(A)\boxtimes_TT^\prime$ is a sheaf on $\CA_{T^\prime}$ with graded free local sections. Hence $\SQ_T(A)$ is an object in $\bB$. 
As $\CZ^{\ol\Lambda}$ is a direct summand of $\CZ$, hence a projective $\CZ$-module, Lemma \ref{lemma-specproj} implies that  $\SQ(A)$ is a projective object in $\bS$.
\end{proof}

%Let $\alpha\in R^+$ and suppose that $T$ is such that $R_T^+=\{\alpha\}$.
\begin{lemma}\label{lemma-transVerma}  Let $T$ and $\Lambda$ be as above and $s\in\hCS$. %Let $\Lambda\subset\CA_{T}$ be a connected component and $A\in\Lambda$. 
\begin{enumerate} 
\item If $\Lambda\ne\Lambda s$, then $\vartheta_s\SV(A)\cong\SV(A)\oplus\SV(As)$.
\item If $\Lambda=\Lambda s$, then $\vartheta_s\SV(A)$ is isomorphic to $\SQ(A^\prime)$ where $A^\prime\in\{A,As\}$ is such that $A^\prime\preceq_T A^\prime s$. In particular, $\vartheta_s\SV(A)$ is projective in $\bS$.  
\end{enumerate}
\end{lemma}
\begin{proof} In the case $\Lambda\ne\Lambda s$ the claim follows immediately from Lemma \ref{lemma-propwc1} and the fact that $\gamma_s^{[\ast]}\SV(A)\cong\SV(As)$. Hence  assume that $\Lambda=\Lambda s$. Let $\mu\in X$ be such that $A,As\in\Lambda_\mu$. 
By Lemma \ref{lemma-firstpropM} we have $\vartheta_s\SV(A)=\vartheta_s\SM_{\{A,As\}}(\CZ^{\pi(A)})\cong\SM_{\{A, As\}}(\CZ\otimes_{\CZ^s}\CZ^{\pi(A)})\cong\SM_{\{A,As\}}(\CZ^{\ol\Lambda})=\SQ(A^\prime)$.
\end{proof}

\begin{corollary}\label{cor-imtrans} Let $T$, $\Lambda$ and   $s$ be as above and suppose that $\Lambda=\Lambda s$. Suppose that $\SM$ is an object in $\bB$  supported on $\Lambda$. Then $\vartheta_s\SM$ splits into a direct sum of shifted copies of objects isomorphic to $\SQ(A)$ with $A\preceq_T As$.  
\end{corollary}
\begin{proof} 
As $\SM$ is an extension of various $\SV(A)$'s with $A\in\Lambda$ the claim follows from Lemma \ref{lemma-transVerma} (2) and the exactness of $\vartheta_s$.  
\end{proof}

\begin{lemma}\label{lemma-split}
Let $T$, $\Lambda$ and   $s$ be as above and suppose that $\Lambda=\Lambda s$. Let $\SM$ be an object in $\bB$ that is supported on $\Lambda$, and let  $A\in\Lambda$ be  such that $A=\alpha\uparrow As$. Then the short exact sequence
$$
0\to(\vartheta_s\SM)_{[\alpha\uparrow A]}\to(\vartheta_s\SM)_{[A,\alpha\uparrow A]}\to(\vartheta_s\SM)_{[A]}\to 0
$$
(cf. Lemma \ref{lemma-reshoms}) splits uniquely. 
\end{lemma}

\begin{proof}
By Corollary \ref{cor-imtrans} $\vartheta_s\SM$ is isomorphic to a direct sum of copies of various $\SQ(B)$'s with $B$ satisfying $B\preceq_T Bs$. In particular, $\SQ(A)$ does not occur since $As\preceq_T A$. The only direct summands that contribute to $(\vartheta_s\SM)_{[A,\alpha\uparrow A]}$ are hence those isomorphic to $\SQ(\alpha\downarrow A)$ and $\SQ(\alpha\uparrow A)$. Now the sequence splits as $\SQ(\alpha\downarrow A)_{[A,\alpha\uparrow A]}= \SQ(\alpha\downarrow A)_{[A]}$ and  $\SQ(\alpha\uparrow A)_{[A,\alpha\uparrow A]}= \SQ(\alpha\uparrow A)_{[\alpha\uparrow A]}$. The uniqueness follows then from the fact that for any object $\SX$ in $\bS$, the sheaf $\SX_{[A]}$ is a sheaf of $\CZ^{\pi(A)}$-modules and  $\SX_{[\alpha\uparrow A]}$ is a sheaf of $\CZ^{\pi(\alpha\uparrow A)}$-modules and $\pi(A)\ne\pi(\alpha\uparrow A)=s_\alpha\pi(A)$. 
\end{proof}

\begin{lemma}\label{lemma-split2} Let $T$ be as above. Suppose $\SN$ is an object in $\bB$ that is isomorphic to a direct sum of various $\SV(A)$'s and extensions $\SX$ that fit into a short exact sequence $0\to \SV(A)\to\SX\to\SV(\alpha\downarrow A)\to 0$.  Suppose that $B$ is maximal in the support of $\SN$, and let $f\colon \SV(B)\to \SN$ be a morphism. Then $f$ splits if and only if  $f|_{[\alpha\downarrow B,B]}\colon \SV(B)\to\SN_{[\alpha\downarrow B,B]}$ splits. 
\end{lemma}
\begin{proof} Note that the maximality of $B$ in the support of $\SN$ implies that there is a direct sum decomposition $\SN=\SC\oplus\SD$ such that $\SC$ is isomorphic to a direct sum of copies of $\SV(B)$ or extensions of $\SV(B)$ with $\SV(\alpha\downarrow B)$, and $\SD$ does not contain $B$ in its support. Then the image of $f$ is contained in $\SC$. Now $f\colon \SV(B)\to\SN$ splits if and only if the corestricted morphism $f\colon \SV(B)\to\SC$ splits. As $\SC_{[\alpha\downarrow B,B]}=\SC$ we analogously deduce that $f\colon \SV(B)\to\SC$ splits if and only if $f_{[\alpha\downarrow B,B]}\colon\SV(A)\to \SC\oplus\SD_{[\alpha\downarrow B,B]}=\SN_{[\alpha\downarrow B,B]}$ splits. 
\end{proof}

\subsection{On the inclusion of standard objects} Suppose $T=S$.  

\begin{lemma}\label{lemma-extstruc} Let $A,B\in\CA$, $A\prec_S B$. Suppose that $f\colon\SV(B)\to\SB(A)_{[B]}$ is  the inclusion of a direct summand. Then there exists some $\alpha\in R^+$, a  sequence $0\to\SV(B)\to\SX\to\SV(\alpha\downarrow B)\to 0$ that is non-split after applying the functor $\cdot\boxtimes_SS^{\alpha}$, and a monomorphism $g\colon\SX\to\SB(A)_{[\alpha\downarrow B,B]}$ such that $f\colon\SV(B)\to\SB(A)_{[B]}\subset\SB(A)_{[\alpha\downarrow B,B]}$ factors over $g$.
\end{lemma}
If we work in a graded situation, the objects $\SV(B)$ and $\SV(\alpha\downarrow B)$ should be placed in appropriate degrees (that we cannot determine, as the proof uses localization). 
\begin{proof} As $A\ne B$, $\SV(B)$ is not a quotient of $\SB(A)|_{\preceq_SB}$. If it were, then it would also be a quotient of $\SB(A)$, and the unicity statement in Theorem \ref{thm-projs} would imply $\SB(A)\cong\SB(B)$, which contradicts $A\ne B$ for reasons of support. As $\SB(A)|_{\preceq B}$ is an extension of shifted copies of various $\SV(C)$'s, there must be an object $\SX$ that fits into a non-split exact sequence
$$
0\to\SV(B)\to\SX\to\SV(C)\to 0\leqno{(\ast\ast)}
$$
and a morphism $g\colon \SX\to\SB(A)|_{{\preceq B}}$ such that $f$ factors over $g$. We can assume that $C$ is maximal with this property. By Lemma \ref{lemma-ses3} (and its proof) we have $C\prec_S B$ and there exists $\alpha\in R^+$ such that $\pi(A)=s_\alpha\pi(C)$ and such that the short exact sequence $(\ast\ast)$ does not split after extension of scalars to $S^{\alpha}$. Moreover, the image of $g$ is contained in $\SB(A)_{[C,B]}\subset \SB(A)|_{{\preceq B}}$.

It remains to  show that $C=\alpha\downarrow B$. First, if $A$ is minimal in a section $\Lambda_\mu$ for some $\mu\in X$, then $\SB(A)\cong\SM_{\Lambda_\mu}(\CZ)$ and $\SB(A)\boxtimes_SS^{\alpha}\cong\SM_{\Lambda_\mu}(\CZ\otimes_SS^{\alpha})$. Then $\CZ\otimes_SS^{\alpha}$ splits into a direct sum of copies of the subgeneric structure algebra, so $\SX$ must be the direct summand of $\SB(A)\boxtimes_SS^{\alpha}$ that contains $B$ in its support. Then $C=\alpha\downarrow B$, as $\alpha\downarrow B$ is the only alcove in the section $\Lambda_\mu$ that lies in the same $\alpha$-string as $B$. 

So let us assume that $A$ is not minimal in any section. In this case, $\SB(A)$ is isomorphic to a direct summand of an object $\SM$ that is obtained by applying a wall crossing functor to an object admitting a Verma flag. Lemma \ref{lemma-transVerma} implies that $\SM\boxtimes_SS^\alpha$ splits into a direct sum of objects that are either isomorphic to some $\SV(C)$ or to a two step extension of such. Hence the same holds for $\SN:=(\SM\boxtimes_SS^{\alpha})_{[C,B]}$. Now $(\SB(A)\boxtimes_SS^{\alpha})_{[C,B]}$ is a direct summand of $\SN$. Now suppose that $C\ne\alpha\downarrow B$. Then we have a morphism 
$$
h:=g\boxtimes_SS^{\alpha}|_{\SV(B)\boxtimes_SS^{\alpha}}\colon \SV(B)\boxtimes_SS^{\alpha}\subset\SX\boxtimes_SS^{\alpha}\to(\SB(A)\boxtimes_SS^{\alpha})_{[C,B]}\subset\SN
$$
such that $h|_{[\alpha\downarrow B,B]}$ is the inclusion of a direct summand (as $\SV(B)|_{[\alpha\downarrow B,B]}=\SX|_{[\alpha\downarrow B,B] }=\SB(A)_{[\alpha\downarrow B,B]}$ by the maximality of $C$). Lemma \ref{lemma-split2} now implies that $h\colon\SV(B)\boxtimes_SS^{\alpha}\to\SN$ splits. This implies that $\SV(B)\boxtimes_SS^{\alpha}\to\SX\boxtimes_SS^{\alpha}$  splits, contrary to our assumption. Hence $C=\alpha\downarrow B$. 
\end{proof}

\section{A functor into the category of Andersen, Jantzen and Soergel}
The main result in this paper is that there exists a functor from the category $\bB$  to the combinatorial category $\bK$ that Andersen, Jantzen and Soergel define in \cite{AJS}. This functor maps the indecomposable projective objects  $\SB(A)$ to the ``special objects'' in $\bK$ that encode the baby Verma multiplicities of projective objects in the category $\bC$ of restricted, $X$-graded representations of a modular Lie algebra (cf. Section \ref{sec-modrep}). 

\subsection{The combinatorial category of Andersen--Jantzen--Soergel}
Recall that $S$ is the symmetric algebra of the $k$-vector space $X^\vee\otimes_\DZ k$. Recall also that we define $S^{\emptyset}=S[\alpha^{\vee-1}\mid \alpha\in R^+]$ and $S^{\alpha}=S[\beta^{\vee-1}\mid \beta\in R^+,\beta\ne\alpha]$ for all $\alpha\in R^+$.

\begin{definition}[\cite{AJS,Soe95}] We define the category  $\bK$ as the category that consists of objects $
M=\left(\{M(A)\}_{A\in{\CA}}, \{M(A,\beta)\}_{A \in {\CA},\beta\in R^+}\right)$, where
\begin{enumerate}
\item  $M(A)$ is an $S^{\emptyset}$-module  for each $A \in {\CA}$ and
\item for $A \in{\CA}$ and $\beta\in R^+$, $M(A ,\beta)$ is an $S^\beta$-submodule of $M(A )\oplus M(\beta\uparrow A )$.
\end{enumerate}
A morphism 
$f\colon M\to N$ in $\bK$
is given by a collection $(f_A)_{A\in{\CA}}$ of homomorphisms $f_A\colon M(A)\to N(A)$ of  $S^\emptyset$-modules, such that for all $A\in{\CA}$ and $\beta\in R^+$, $f_A\oplus f_{\beta\uparrow A}$ maps $M(A,\beta)$ into $N(A,\beta)$.
\end{definition}

\subsection{The functor $\Psi\colon\bB\to\bK$}
We now fix the base ring $T=S$ and consider the category $\bB=\bB_S$. For simplicity we write $\preceq$ instead of $\preceq_S$.
Let $\SM$ be an object in $\bB$. In order to simplify notation write 
$\SM^\emptyset:=\SM\boxtimes_SS^{\emptyset}$ and $\SM^\alpha:=\SM\boxtimes_SS^{\alpha}$ for $\alpha\in R^+$. Define the  object $\Psi(\SM)$ of $\bK$ in the following way. For $A\in\CA$ and $\alpha\in R^+$,  
\begin{align*}
\Psi(\SM)(A)&:=\Gamma(\SM^{\emptyset}_{[A]}),\\
\Psi(\SM)(A,\alpha)&:=\Gamma(\SM^{\alpha}_{[\{A,\alpha\uparrow A\}]}).
\end{align*}
A few remarks are in order. Lemma \ref{lemma-upmin} shows that  $\{A,\alpha\uparrow A\}$ is locally closed in $\CA_{S^{\alpha}}$. As $A$ and $\alpha\uparrow A$ are in distinct $\DZ R$-orbits (i.e. in distinct connected components of $\CA_{S^\emptyset}$), the  short exact sequence
$$
0\to\SM^{\alpha}_{[\alpha\uparrow A]}\to \SM^{\alpha}_{[\{A,\alpha\uparrow A\}]}\to \SM^{\alpha}_{[A]}\to 0
$$
splits canonically after applying the functor $\boxtimes_{S^\alpha}S^{\emptyset}$. Hence 
$$
\Gamma(\SM^{\alpha}_{[\{A,\alpha\uparrow A\}]})\otimes_{S^\alpha}S^\emptyset=\Gamma(\SM^{\alpha}_{[\{A,\alpha\uparrow A\}]}\boxtimes_{S^\alpha}S^\emptyset)=\Gamma(\SM^{\emptyset}_{[A]})\oplus\Gamma(\SM^{\emptyset}_{[\alpha\uparrow A]})
$$
canonically. In particular, we can consider $\Psi(\SM)(A,\alpha)$ as a subset of $\Psi(\SM)(A)\oplus\Psi(\SM)(\alpha\uparrow A)$. So the above indeed yields a functor $\Psi$ from $\bB$ to $\bK$.

\begin{remark}\label{rem-rescomb} Via the  natural inclusions 
$$
\Gamma(\SM^\alpha_{[\alpha\uparrow A]})\subset\Gamma(\SM^\alpha_{[\{A,\alpha\uparrow A\}]})\subset \Gamma(\SM^{\emptyset}_{[\{A,\alpha\uparrow A\}]})=\Gamma(\SM^\emptyset_{[A]})\oplus \Gamma(\SM^\emptyset_{[\alpha\uparrow A]})
$$ 
we can identify $\Gamma(\SM^{\alpha}_{[\alpha\uparrow A]})$ with $\Gamma(\SM^{\alpha}_{[\{A,\alpha\uparrow A\}]})\cap \Gamma(\SM^{\emptyset}_{[\alpha\uparrow A]})$. \end{remark}

\subsection{Wall crossing functors on $\bK$}

We fix a simple reflection $s\in\hCS$. In this section we recall the definition of the wall crossing functor $\vartheta^c_s\colon\bK\to\bK$ of \cite{AJS} in the normalization of \cite{FieJAMS}. 
For $M\in\bK$, $B\in\CA$ and
$\alpha\in R^+$ set
\begin{eqnarray*}
\vartheta^c_s M({A}) &:=  &  M({A})\oplus M({As}), 
\\
\vartheta^c_s M({A},\alpha) &:=  &
\begin{cases}  
\left\{(\alpha^\vee x+y,y)\mid x,y\in M(A,\alpha)\right\}, & \text{if $As=\alpha\uparrow A$},\\
\alpha^\vee\cdot M(\alpha\uparrow A,\alpha)\oplus M(\alpha\downarrow {A},\alpha), & \text{if $As=\alpha\downarrow A$},\\
M(A,\alpha)\oplus M(As,\alpha) , & \text{if $As\ne\alpha\uparrow A,As\ne\alpha\downarrow A$}
\end{cases}
\end{eqnarray*}
with the obvious inclusion $\vartheta_s^c M(A,\alpha)\subset\vartheta_s^cM(A)\oplus\vartheta_s^cM(\alpha\uparrow A)$. 
The following result is the main ingredient in the proof that the indecomposable projective objects in $\bB$ correspond, via $\Psi$, to the special objects considered by Andersen, Jantzen and Soergel.
\begin{proposition}\label{prop-psiint} For any $s\in\hCS$ there is an isomorphism
$$
\Psi\circ\vartheta_s\cong\vartheta_s^c\circ\Psi
$$
of functors from $\bB$ to $\bK$.
\end{proposition}
\begin{proof} Let $\SM$ be an object in $\bB$. As the wall crossing functor commutes with base change by \cite[Lemma 7.2]{FieLanWallCross},   there is a functorial isomorphism $(\vartheta_s \SM)^{\emptyset}_{[A]}\cong\vartheta_s(\SM^{\emptyset})_{[A]}$ and Lemma  \ref{lemma-propwc1} yields a further isomorphism onto $\SM^{\emptyset}_{[A]}\oplus\gamma_s^{[\ast]}(\SM^{\emptyset}_{[As]})$ for any $A\in\CA$. Hence we obtain a functorial isomorphism (of $S^{\emptyset}$-modules)
\begin{align*}
\Psi(\vartheta_s\SM)(A)&=\Gamma((\vartheta_s\SM)^\emptyset_{[A]})\\
&\cong\Gamma(\SM^{\emptyset}_{[A]})\oplus\Gamma(\SM^{\emptyset}_{[As]})=\vartheta_s^c(\Psi\SM)(A).
\end{align*}

Let $\alpha\in R^+$ and $A\in\CA$ and denote by $\Lambda$ the connected component that contains $A$. If $As\not\in\{\alpha\uparrow A,\alpha\downarrow A\}$, then $\Lambda\ne\Lambda s$ by Lemma \ref{lemma-Asuparrow}.  In this case we obtain as before  canonical isomorphisms $(\vartheta_s\SM)^\alpha_{[\{A,\alpha\uparrow A\}]}\cong\vartheta_s(\SM^\alpha)_{[\{A,\alpha\uparrow A\}]}\cong\SM^{\alpha}_{[\{A,\alpha\uparrow A\}]}\oplus\gamma_s^{[\ast]}(\SM^{\alpha}_{[\{As,\alpha\uparrow As\}]})$ and hence
\begin{align*}
(\Psi\vartheta_s\SM)( A,\alpha)&=\Gamma((\vartheta_s\SM)^\alpha_{[\{A,\alpha\uparrow A\}]})\\
&\cong\Gamma(\SM^\alpha_{[\{A,\alpha\uparrow A\}]})\oplus\Gamma(\SM^\alpha_{[\{As,\alpha\uparrow As\}]}) =(\vartheta_s^c\Psi\SM)(A,\alpha)
\end{align*}  
canonically (note that $(\alpha\uparrow A)s=\alpha\uparrow As$ in this case). 

Now suppose that $As=\alpha\uparrow A$. Then $\Lambda=\Lambda s$ and  $\{A,\alpha\uparrow A\}$ is a locally closed $s$-invariant subset of $\CA_{S^{\alpha}}$. Hence
\begin{align*}
(\Psi\vartheta_s\SM)(A,\alpha)&=\Gamma(
(\vartheta_s\SM^{\alpha})_{[\{A,\alpha\uparrow A\}]})\\
&=\Gamma(\vartheta_s(\SM^{\alpha}_{[\{A,\alpha\uparrow A\}]}))\text{ by Proposition \ref{prop-propwc2}},\\
&=\CZ\otimes_{\CZ^s}\Gamma(\SM^{\alpha}_{[\{A,\alpha\uparrow A\}]})\text{ by the defining property of $\vartheta_s$}\\
&=\CZ^{\ol\Lambda}\otimes_{\CZ^{\ol\Lambda, s}}\Gamma(\SM^{\alpha}_{[\{A,\alpha\uparrow A\}]})\text{ by Lemma \ref{lemma-transsinv}}\\
&=\left\{(\alpha^\vee x+y,y)\mid x,y\in \Gamma(\SM^{\alpha}_{[\{A,\alpha\uparrow A\}]})\right\}\text{ by  \cite[Lemma 4.2]{FieLanWallCross}}\\
&=(\vartheta_s^c\Psi\SM)(A,\alpha).
\end{align*}

Finally suppose that $As=\alpha\downarrow A$. By Lemma \ref{lemma-split} we have
$$
(\vartheta_s\SM)^\alpha_{[\{A,\alpha\uparrow A\}]}
=\vartheta_s(\SM^{\alpha})_{[\{A,\alpha\uparrow A\}]}\cong\vartheta_s(\SM^{\alpha})_{[A]}\oplus\vartheta_s(\SM^{\alpha})_{[\alpha\uparrow A]}
$$
canonically. From the previous step we obtain identifications
\begin{align*}
\Gamma((\vartheta_s\SM)^{\alpha}_{[A]})&=\alpha^\vee\Gamma(\SM^{\alpha}_{[\{\alpha\downarrow A,A\}]}) =\alpha^\vee\Psi\SM(\alpha\downarrow A,\alpha)
\end{align*}
and
\begin{align*}
\Gamma((\vartheta_s\SM)^{\alpha}_{[\alpha\uparrow A]})&=\Gamma(\SM^{\alpha}_{[\{\alpha\uparrow A,\alpha\uparrow^2 A\}]})=\Psi\SM(\alpha\uparrow A,\alpha).
\end{align*}
Hence $(\Psi\vartheta_s\SM)(A,\alpha)=(\vartheta_s^c\Psi\SM)(A,\alpha)$ also in this case. 
\end{proof}

\subsection{Indecomposability} 
Now we show that indecomposable projective objects in $\bB$ stay indecomposable after applying the functor $\Psi$.
\begin{lemma}\label{lemma-Psiindec}  For any $A\in\CA$, the object $\Psi\SB(A)$ is indecomposable in $\bK$.
\end{lemma}
\begin{proof} Suppose  $\Psi\SB(A)=M\oplus N$ is a direct sum decomposition  in $\bK$. As $\Psi\SB(A)(A)=\Gamma(\SB(A)^\emptyset_{[A]})$ is a free $S^{\emptyset}$-module of rank 1, we have $M(A)\oplus N(A)\cong S^{\emptyset}$. Hence we can assume $M(A)\cong S^{\emptyset}$ and  $N(A)=0$. If $N$ is not the zero object, then we can choose a $\preceq_S$-minimal $B\in\CA$  such that $N(B)\ne 0$.  So $A\ne B$.

Consider now the  $S$-module  $\Gamma(\SB(A)_{[B]})$. As $\SB(A)_{[B]}$ is root reflexive, this is the intersection of the $S^{\alpha}$-modules $\Gamma(\SB(A)^\alpha_{[B]})$ inside $\Gamma(\SB(A)^\emptyset_{[B]})=M(B)\oplus N(B)$. 
For any positive root $\alpha$ we  have a direct sum decomposition 
$$
\Gamma(\SB(A)_{[\{\alpha\downarrow B,B\}]}^\alpha)=\Psi(\SB(A))(\alpha,\alpha\downarrow B)=M(\alpha,\alpha\downarrow B)\oplus N(\alpha,\alpha\downarrow B).
$$
By Remark \ref{rem-rescomb}, we obtain $\Gamma(\SB(A)^{\alpha}_{[B]})$ by intersecting $\Gamma(\SB(A)^{\alpha}_{[\{\alpha\downarrow B,B\}]})$ with $\Gamma(\SB(A)^\emptyset_{[B]})=M(B)\oplus N(B)$. Hence we obtain a direct sum decomposition 
$$
\Gamma(\SB(A)^\alpha_{[B]})=\left(\Gamma(\SB(A)^\alpha_{[B]})\cap M(B)\right)\oplus\left(\Gamma(\SB(A)^\alpha_{[B]})\cap N(B)\right)
$$
for any positive root $\alpha$. Taking the intersection over all $\alpha$ yields a direct sum decomposition
$$
\Gamma(\SB(A)_{[B]})=\left(\Gamma(\SB(A)_{[B]})\cap M(B)\right)\oplus\left(\Gamma(\SB(A)_{[B]})\cap N(B)\right)
$$
inside $\Gamma(\SB(A)_{[B]}^\emptyset)=M(B)\oplus N(B)$. Using  Proposition \ref{prop-equivcat} we arrive at a direct sum decomposition $\SB(A)_{[B]}=\SM\oplus\SN$ in $\bS$ with $\Psi(\SM)(B)=M(B)$ and $\Psi(\SN)(B)=N(B)$. 

As direct summands of $\SB(A)_{[B]}$, both $\SM$ and $\SN$ admit a Verma flag. Clearly both sheaves are supported on $B$.  There is hence some $n\in\DZ$ and an inclusion   $f\colon \SV(B)[n]\to\SB(A)_{[B]}$ of a direct summand  such that the image is contained in $\SN\subset\SB(A)_{[B]}\subset \SB(A)|_{\preceq_S B}$. Lemma \ref{lemma-extstruc} shows that there is some $\alpha\in R^+$ and an extension $\SX$ of $\SV(B)$ and $\SV(\alpha\downarrow B)$ that is non-split even after applying $\cdot\boxtimes_SS^{\alpha}$ and such that  $f$ factors over a morphism $g\colon\SX\to\SB(A)|_{\preceq_S B}$. Now $\Psi(\SX)(\alpha,\alpha\downarrow B)=\Gamma(\SX^{\alpha}_{[\{\alpha\downarrow B,B\}]})$ is non-split by Proposition \ref{prop-equivcat}. But $\Psi(\SX)(\alpha,\alpha\downarrow B)\subset \Gamma(\SX^{\emptyset}_{[\alpha\downarrow B]})\oplus\Gamma(\SX^\emptyset_{[B]})$, and the direct summand on the left is contained in $M(\alpha\downarrow B)$, while the direct summand on the right is a direct summand of $N(B)$. This contradicts the fact that $\Psi(\SX)$ is non-split. Hence $N=0$ and $\Psi\SB(A)$ is indecomposable.
\end{proof}

\subsection{The special objects of Andersen, Jantzen and Soergel}
Let $\mu\in X$ and let $A_\mu^-$ be the $\preceq_S$-minimal element in $\CA_\mu$ (note that $\CA$ is the only connected component of $\CA_S$, so $\CA_\mu$ is, by the earlier definition, the set of alcoves that contain $\mu$ in their closure). Define the object $Q_\mu\in\bK$ as follows.
\begin{align*}
Q_\mu(A) & :=  
\begin{cases} S^\emptyset, & \text{if $A\in\CW_\mu(A_\mu^-)$}, \\
0, & \text{else},
\end{cases}
 \\
Q_\mu(A,\alpha) & :=  
\begin{cases} S^\alpha, & \text{ if $A\in\CA_\mu$ and $\alpha\uparrow A\not\in\CA_\mu$},\\
\left\{(x,y)\in S^{\alpha}\oplus S^{\alpha}\left|\, \begin{matrix}x\equiv y\\ \mod\alpha^\vee\end{matrix}\right\}\right.,& \text{ if $A,\alpha\uparrow A\in\CA_\mu$},\\
\alpha^\vee S^\alpha, & \text{ if $A\not\in\CA_\mu$, $\alpha\uparrow A\in\CA_\mu$},\\
0, & \text{ else},
\end{cases}
\end{align*}
with the obvious inclusions $Q_\mu(A,\alpha)\subset Q_\mu(A)\oplus Q_\mu(\alpha\uparrow A)$ for all $A\in\CA$ and $\alpha\in R^+$. 

\begin{definition} Denote by $\bM^\circ$ the full subcategory of $\bK$ that contains all objects that are isomorphic to a direct sum of direct summands of objects of the form $\vartheta_s^c\cdots\vartheta_t^c Q_\mu$ for all sequences $s,\dots, t\in\hCS$ and all $\mu\in X$. 
\end{definition}
The objects in $\bM^{\circ}$ are sometimes called the {\em special objects}. 
One of the main results in the book \cite{AJS} (cf. also the overview article \cite{Soe95}) is the construction of a functor $\mathbb V$ from the deformed category $\CC$ of certain $X$-graded $\fg$-modules, where $\fg$ is the  $k$-Lie algebra associated with $R$  (cf. Section \ref{sec-modrep}), into the category $\bK$.  The functor $\DV$ is fully faithful on the subcategory of deformed representations that are projective over the deformation ring (\cite[Theorem 4.2.3]{Soe95}), and it commutes with wall crossing functors (\cite[Theorem 4.3.1]{Soe95}). The subcategory of projective objects in $\CC$ can be obtained by applying  wall crossing  functors to special projective objects (associated to special alcoves). Moreover, each projective in $\CC$ admits a baby Verma flag, hence admits an epimorphism onto some baby Verma module. Finally,  each projective is free over the deformation ring. Hence the classification of projective objects in $\CC$  yields the following classification of the indecomposable special objects.  
\begin{proposition}\label{prop-classspec}  For all $A\in\CA$ there is an up to isomorphism  unique object $Q(A)\in\bM^\circ$ with the following properties:
\begin{itemize}
\item $Q(A)$ is indecomposable in $\bK$,
\item $Q(A)(B)=0$ unless $A\preceq_S B$, and $Q(A)(A)=S^\emptyset$.
\end{itemize}
\end{proposition}

For example, $Q(A_\mu^-)=Q_\mu$. 
Now we  obtain the connection to the sheaves on the alcoves.

\begin{theorem}\label{thm-Bspec} 
For all $A\in\CA$ we have an isomorphism $\Psi(\SB(A))\cong Q(A)$ in $\bK$. Moreover, $\rk_{S^\emptyset}\, Q(A)(B)=(\SB(A):\SV(B))(1)$ for all $B\in\CA$.
\end{theorem}
Here, $\rk_{S^\emptyset}\, M$ is the rank of a free $S^{\emptyset}$-module, and by $(\SB(A):\SV(B))(1)$ we mean the polynomial $(\SB(A):\SV(B))$ evaluated at $1$. 
\begin{proof} Suppose that $A$ is special, i.e. $A=A_\mu^-$. Then one checks immediately that $\Psi(\SB(A_\mu^-))=\Psi(\SM_{\CA_\mu}(\CZ))\cong Q_\mu=Q(A_\mu^-)$. As $\Psi$ intertwines the wall crossing functors on $\bB$ and on $\bK$ by Proposition \ref{prop-psiint}, we obtain, using the wall crossing algorithm that constructs the indecomposable projectives from Theorem \ref{thm-projs},  that $\Psi(\SB(A))$ is contained in $\bM^{\circ}$. This is an indecomposable object by Lemma \ref{lemma-Psiindec}. Remark \ref{rem-strucproj} now yields that $\Psi(\SB(A))(A)\cong S^{\emptyset}$ and $\Psi(\SB(A))(B)=0$ unless $A\preceq_S B$. Hence $\Psi(\SB(A))\cong Q(A)$.  Moreover, 
\begin{align*}
\rk_{S^\emptyset}\,\Psi(\SB(A))(B)&=\rk_{S^\emptyset}\,\Gamma(\SB(A)^{\emptyset}_{[B]})\\
&=(\SB(A)^\emptyset:\SV(B)^\emptyset)\\
&=(\SB(A):\SV(B)).
\end{align*}\end{proof}

\section{The application to modular representation theory}\label{sec-modrep}
In this section we quickly review some of the main results in \cite{AJS}. One of the applications of Theorem \ref{thm-Bspec} is that the indecomposable projective objects $\SB(A)$'s encode rational simple characters of the reductive groups associated with $R$, if the characteristic of $k$ is larger than the Coxeter number. We assume in this section that $k$ is algebraically closed and that its characteristic is a prime number $p$ that is larger than the Coxeter number of $R$.

Let $\fg$ be the Lie algebra over $k$. Denote by $\bC$ the category of $X$-graded restricted representations of $\fg$ (cf. \cite{Soe95, FieBull}). For any $\lambda$ there is the simple highest weight module $L(\lambda)$ and the baby Verma module $Z(\lambda)$ in $\bC$. 
The characters of the baby Verma modules are easy to determine, and one can calculate the characters of the simple modules once the Jordan--H\"older multiplicities $[Z(\lambda):L(\mu)]$ are known.  Now consider the map $m_p\colon X\to X$, $m_p(v)=pv$. There is a unique $\hCW$-action $\cdot_p$ of $\hCW$ on $X$ such that $m_p(wv)=w\cdot_pm_p(v)$ for all $w\in\hCW$ and $v\in X$. Using the linkage and the translation principle it turns out that one needs to know only the numbers $[Z(w\cdot_p0):L(x\cdot_p0)]$ for $w,x\in\hCW$, where $\rho=\frac{1}{2}\sum_{\alpha\in R^+}\alpha$. In \cite{AJS} it is shown:

\begin{theorem} \label{thm-mainAJS} For $w,x\in\hCW$ let $A,B\in\CA$ be the unique elements with $x(\rho)\in A$ and $w(\rho)\in B$. Then
$$
[Z(w\cdot_p0):L(x\cdot_p0)]=\rk_{S^\emptyset}\, Q(A)(B).
$$
\end{theorem}

Theorem \ref{thm-Bspec} then yields the following.
\begin{theorem} For $A,B\in\CA$ let $w,x\in\hCW$ be the unique elements with $x(\rho)\in A$ and $w(\rho)\in B$. Then $$
[Z(w\cdot_p0):L(x\cdot_p0)]=(\SB(A):\SV(B))(1).$$
\end{theorem}

%For us, the above result was  the main motivation to construct the categories $\bB$ and $\bS$ 

\end{document}